\documentclass[12pt,oneside,reqno]{amsart}

\usepackage{amsmath,amssymb,amsthm,textcomp,mathtools}
\usepackage{amsfonts,graphicx}
\usepackage[mathscr]{eucal}
\usepackage{color}
\usepackage{csquotes}
\usepackage{enumitem}
\allowdisplaybreaks

\numberwithin{equation}{section}

\theoremstyle{definition}

\numberwithin{equation}{section}

\newtheorem{theorem}{\bf Theorem}[section]
\newtheorem{remark}{\bf Remark}[section]
\newtheorem{proposition}{Proposition}[section]

\newtheorem{definition}{Definition}[section]
\newtheoremstyle
{remarkstyle}
{}
{11pt}
{}
{}
{\bfseries}
{:}
{     }
{\thmname{#1} \thmnumber{#2} }

\theoremstyle{remarkstyle}

\begin{document}
\title{On Multiparameter Generalized Counting Process and its Time-Changed Variants}
\author[Manisha Dhillon]{Manisha Dhillon}
\address{Manisha Dhillon, Department of Mathematics, Indian Institute of Technology Bhilai, Durg 491002, India.}
\email{manishadh@iitbhilai.ac.in}
\author[Kuldeep Kumar Kataria]{Kuldeep Kumar Kataria}
\address{Kuldeep Kumar Kataria, Department of Mathematics, Indian Institute of Technology Bhilai, Durg 491002, India.}
\email{kuldeepk@iitbhilai.ac.in}
\date{\today}
\begin{abstract}
We introduce and study a multiparameter version of the generalized counting process (GCP), where there is a possibility of finitely many arrivals simultaneously. We call it the multiparameter GCP. In a particular case, it is uniquely represented as a weighted sum of independent multiparameter Poisson processes. For a specific case, we establish a relationship between the multiparameter GCP and the sum of independent GCPs. Some of its time-changed variants are studied where the time-changing components used are the multiparameter stable subordinator and the multiparameter inverse stable subordinator. An integral of the multiparameter GCP is defined, and its asymptotic distribution is obtained. 
\end{abstract}

\maketitle

Multiparameter random processes are extension of one-parameter stochastic processes that intersect with various other branches of mathematics and physics. In \cite{Khoshnevisan2002a}, their connections to group theory, analytic number theory, real and functional analysis are discussed. The multiparameter random processes have potential applications in various fields, for example, mathematical statistics, statistical mechanics and cognitive science, \textit{etc.} (see \cite{Cao1999}). For more details on multiparameter random process, we refer the reader to \cite{Barndorff2001}, \cite{Merzbach1986}, and references therein.

The presence of fractional derivatives in a dynamical system induces a global memory effect in it. So, the time-changed random processes commonly known as the fractional random processes are of particular interest in the theory of stochastic processes. In the past two decades, the time-changed Poisson processes have been extensively studied (see \cite{Beghin2009},  \cite{Leskin2003}, \cite{Meerschaert2011}). The one dimensional distributions of these time-changed processes are governed by a system of differential equations involving fractional derivatives. For a similar study on the time-changed Poisson random fields indexed by $\mathbb{R}_+^2$, we refer the reader to \cite{Beghin2020} and \cite{Kataria2024}.

Di Crescenzo {\it et} {\it al.} \cite{Di Crescenzo2016} introduced and studied a L\'evy process $\{M(t)\}_{t\geq0}$, namely, the generalized counting process (GCP). It is a generalization of the Poisson process in the sense that it performs independently $k$ kinds of jumps of amplitude $1,2,\dots,k$ with positive rates $\lambda_1,\lambda_2,\dots,\lambda_k$, respectively.  Its state probabilities $p(n,t)=\mathrm{Pr}\{M(t)=n\}$ solve the following system of differential equations:
\begin{equation*}
	\frac{\mathrm{d}}{\mathrm{d}t}p(n,t)=-\sum_{j=1}^{k}\lambda_j p(n,t)+\sum_{j=1}^{n\wedge k}\lambda_j p(n-j,t), \ n\geq0,
\end{equation*}
with initial conditions $p(0,0)=1$ and $p(n,0)=0$ for all $n\geq1$. Here, $n\wedge k=\min\{n,k\}$. Its state probabilities are given by 
\begin{equation*}
	p(n,t)=\sum_{\Omega(k,n)}\prod_{j=1}^{k}\frac{(\lambda_{j}t)^{x_{j}}}{x_{j}!}e^{-\lambda_j t},\ \ n\ge 0,
\end{equation*}
where $\Omega(k,n)=\{(x_{1},x_{2},\dots,x_{k}):\sum_{j=1}^{k}jx_{j}=n,\ x_{j}\in\mathbb{N}_0\}$. 
Thus, the GCP is equal in distribution to a weighted sum of $k$ independent Poisson processes (see \cite{Kataria2022a}). That is, 
\begin{equation*}
	M(t)\overset{d}{=}\sum_{j=1}^{k}jN_j(t),
\end{equation*}
where $\{N_j(t)\}_{t\geq0}$'s are independent Poisson process with intensity parameter $\lambda_j>0$. Here, $\overset{d}{=}$ denotes equal in distribution. For more details on GCP and its time-changed variants, we refer the reader to \cite{Dhillon2024}, \cite{Kataria2022a}, \cite{Kataria2022b}, \cite{Kataria2022}, and references therein. 

Recently, Vishwakarma and Kataria \cite{Vishwakarma2025} introduced and studied a multiparameter version of the Poisson process. For a particular case, it is shown that the multiparameter Poisson process has a unique representation as the sum of independent one parameter Poisson processes. Its integral is studied, and its several martingale characterizations are obtained. 

In this paper, we introduce and study a multiparameter version of the GCP and its time-changed variants. The paper is organized as follows: In Section \ref{pre}, we give some preliminary results and set notations that will be used later. In Section \ref{sec3}, we define a multiparameter generalized counting process (multiparameter GCP) with index set $\mathbb{R}^d_+$, $d\ge1$. It is a multiparameter counting process with independent and stationary increments whose one dimensional distributions coincide with that of a GCP. Its various distributional properties such as the one dimensional distributions, probability generating function (pgf) and their governing system of differential equations are derived.  Its integral over rectangle in $\mathbb{R}^d_+$ is defined and its unique representation is established. An asymptotic distribution of this integral process is obtained. In Section \ref{sec4}, we study two time-changed variants of the multiparameter GCP where the time changing components are multiparameter stable subordinator and multiparameter inverse stable subordinator. The explicit forms of their one dimensional distributions, means and variances are derived. 

\section{Preliminaries}\label{pre}
In this section, we give some preliminary details on Mittag-Leffler function, fractional derivatives, and stable subordinator, multiparameter stable subordinator and their inverse.

First, we set some notations that will be used throughout the paper.
\subsection{Notations}  Let $\mathbb{R}_+$ denote the set of non-negative real numbers, that is, $\mathbb{R}_+=[0,\infty)$, and let $\mathbb{R}^d_+$, $d\ge1$ be the set of $d$-dimensional vectors with non-negative real entries. Also, let $\mathbb{N}_0=\mathbb{N}\cup\{0\}$. 
\begin{itemize}
	\item The random vector and the real vector are denoted by bold uppercase alphabets and bold lowercase alphabets, respectievly. 
	The zero vector in $\mathbb{R}^d_+$ is denoted by $\textbf{0}$, and $\textbf{1}=(1,1,\dots,1)$ is a $d$-dimensional vector.
	\item The dot product of two elements $\textbf{t}=(t_1,t_2,\dots,t_d)$ and $\textbf{s}=(s_1,s_2,\dots,s_d)$ in $\mathbb{R}^d_+$ is the usual dot product on $\mathbb{R}^d_+$, that is, $\textbf{t}\cdot \textbf{s}=\sum_{i=1}^{d}t_is_i$. Also, the difference of these two elements is their component-wise difference, that is, $\textbf{t}-\textbf{s}=(t_1-s_1,t_2-s_2,\dots,t_d-s_d)$.
	
	\item Let $c\in\mathbb{R}$ and $\textbf{t}\in\mathbb{R}^d_+$. Then, the scalar product $c\textbf{t}$ is the standard component-wise product, that is,   $c\textbf{t}=(ct_1,ct_2,\dots,ct_d)$.
	
	\item  For any $\textbf{t}\in \mathbb{R}^d_+$ and $\textbf{s}\in\mathbb{R}^d_+$, let $\textbf{s}\preceq \textbf{t}$ denote $s_i\leq t_i$ for all $i\in\{1,2,\dots,d\}$. That is, the usual partial ordering on $\mathbb{R}^d_+$ is considered. 
	
\end{itemize}
\subsection{Mittag-Leffler function}
The three-parameter Mittag-Leffler function is defined as (see \cite[Eq. (1.9.1)]{Kilbas2006})
\begin{equation}\label{mitag}
	E_{\alpha,\beta}^{\gamma}(x)\coloneqq\frac{1}{\Gamma(\gamma)}\sum_{j=0}^{\infty} \frac{\Gamma(j+\gamma)x^{j}}{j!\Gamma(j\alpha+\beta)},\ x\in\mathbb{R},
\end{equation}
where $\alpha>0$, $\beta>0$ and $\gamma>0$. For $\gamma=1$, it reduces to the two-parameter Mittag-Leffler function. Further, for $\gamma=\beta=1$, we get the Mittag-Leffler function. 
\subsection{Fractional derivatives}
For $\gamma\geq 0$, the Riemann-Liouville fractional derivative is defined as follows (see \cite{Kilbas2006}):
\begin{equation*}
	\partial_t^\gamma f(t):=\left\{
	\begin{array}{ll}
		\dfrac{1}{\Gamma{(m-\gamma)}}\displaystyle \frac{\mathrm{d}^{m}}{\mathrm{d}t^{m}}\int^t_{0} \frac{f(s)}{(t-s)^{\gamma+1-m}}\,\mathrm{d}s,\ m-1<\gamma<m,\vspace{.2cm}\\
		\displaystyle\frac{\mathrm{d}^{m}}{\mathrm{d}t^{m}}f(t),\  \gamma=m,
	\end{array}
	\right.
\end{equation*}
where $m$ is a positive integer.

The Caputo fractional derivative is defined as follows (see \cite{Kilbas2006}):
\begin{equation}\label{caputo}
	\mathcal{D}_t^\alpha f(t)\coloneqq\left\{
	\begin{array}{ll}
		\dfrac{1}{\Gamma{(1-\alpha)}}\displaystyle\int^t_{0} (t-s)^{-\alpha}f'(s)\,\mathrm{d}s, \ 0<\alpha<1,\\\\
		f'(t),\ \alpha=1.
	\end{array}
	\right.
\end{equation}
The following result holds true (see \cite[Eq. (3.1)]{Meerschaert2013}): 
\begin{equation*}
	\mathcal{D}_t^\alpha f(t)=\partial_t^\alpha f(t)-\frac{t^{-\alpha}}{\Gamma(1-\alpha)}f(0^+),\ \alpha\in(0,1).
\end{equation*}

\subsection{Stable subordinator and its inverse}\label{alphastab}
A subordinator $\{D^\alpha(t)\}_{t\ge0}$, $\alpha\in(0,1)$ is called $\alpha$-stable if its Laplace transform is given by $\mathbb{E}e^{-wD^\alpha(t)}=e^{-tw^\alpha}$, $w>0$. Its first passage time process $\{L^\alpha(t)\}_{t\ge0}$, that is,
\begin{equation*}
	L^\alpha(t)\coloneqq\inf\{u>0:D^\alpha(u)>t\}
\end{equation*}
is called the inverse $\alpha$-stable subordinator. The Laplace transform of its density $l_\alpha(x,t)=\mathrm{Pr}\{{L}^\alpha(t)\in\mathrm{d}x\}$, $x\ge0$ with respect to time variable is given by (see \cite{Meerschaert2011})
\begin{equation}\label{lapminv}
	\int_{0}^{\infty}e^{-wt}l_\alpha(x,t)\ \mathrm{d}t=w^{\alpha-1}e^{-w^{\alpha}x},\ w>0.
\end{equation} 
It is the solution of the following fractional differential equation (see \cite{Beghin2020}):
\begin{equation}\label{invsequ}
	\partial_t^\alpha l_\alpha(x,t)=-\partial_xl_\alpha(x,t),
\end{equation}
with $l_\alpha(0,t)=t^{-\alpha}/\Gamma(1-\alpha)$, $t\ge0$.

For $i=1,2,\dots,d$, let $\{S_i^{\alpha_i}(t_i) \}_{t_i\ge0}$, $\alpha_i\in(0,1)$ be independent stable subordinators and  $\{L_i^{\alpha_i}(t_i)\}_{t_i\ge0}$ be  independent inverse stable subordinators. Then,
a multivariate $d$-parameter process $\{\boldsymbol{\mathscr{D}}_{\boldsymbol{\alpha}}(\textbf{t})\}_{\textbf{t}\in\mathbb{R}^d_+}$, $\boldsymbol{\alpha}=(\alpha_1,\alpha_2,\dots,\alpha_d)$ such that
\begin{equation}\label{mssdef}
	\boldsymbol{\mathscr{D}}_{\boldsymbol{\alpha}}(\textbf{t})=(D_1^{\alpha_1}(t_1),D_2^{\alpha_2}(t_2),\dots,D_d^{\alpha_d}(t_d)),\ \textbf{t}=(t_1,t_2,\dots,t_d)\in\mathbb{R}^d_+,
\end{equation}
is called multiparameter stable subodinator. And a multivariate $d$-parameter process $\{\boldsymbol{\mathscr{L}}_{\boldsymbol{\alpha}}(\textbf{t})\}_{\textbf{t}\in\mathbb{R}^d_+}$ such that
\begin{equation}\label{missdef}
	\boldsymbol{\mathscr{L}}_{\boldsymbol{\alpha}}(\textbf{t})=(L_1^{\alpha_1}(t_1),L_2^{\alpha_2}(t_2),\dots,L_d^{\alpha_d}(t_d)),\ \textbf{t}=(t_1,t_2,\dots,t_d)\in\mathbb{R}^d_+,
\end{equation}
is called multiparameter inverse stable subordinator.
\section{Multiparameter generalized counting process}\label{sec3}
Here, we introduce and study a multiparameter generalized counting process with indexing set $\mathbb{R}^d_+$, $d\ge1$. We call it the multiparameter GCP. For fix positive vectors  $\boldsymbol{\Lambda}_j=(\lambda_{j1},\lambda_{j2},\dots,\lambda_{jd})\succ \textbf{0}$, $j=1,2,\dots,k$ and $\mathbf{t}=(t_1,t_2,\dots,t_d)\in\mathbb{R}^{d}_{+}$, let us consider the following sequence $\{p(n,\mathbf{t})\}_{n\ge0}$ of real numbers: 
\begin{equation}\label{MGCPMF}
	p(n,\mathbf{t})=
	\sum_{\Omega(k,n)}\prod_{j=1}^{k}\frac{(\mathbf{\Lambda}_j \cdot\mathbf{t})^{x_{j}}}{x_j!}e^{-\mathbf{\Lambda}_j\cdot\mathbf{t}},\ n\ge0,
\end{equation}
where $\Omega(k,n)=\{(x_1,x_2,\dots,x_k):\sum_{j=1}^{k}jx_j=n,\, x_j\in\mathbb{N}_0\}$ and $\mathbf{\Lambda}_j\cdot\mathbf{t}=\lambda_{j1}t_1+\lambda_{j2}t_2+\dots+\lambda_{jd}t_d$.
Note that $\{p(n,\mathbf{t})\}_{n\ge0}$ satisfies the regularity condition, that is, $\sum_{n=0}^{\infty}p(n,\mathbf{t})=1$.
So, for each $\textbf{t}\in\mathbb{R}^d_+$, it is the pmf of GCP with mean $\sum_{j=1}^{k}j\boldsymbol{\Lambda}_j\cdot\textbf{t}$.
\begin{definition}(\textbf{Multiparameter counting process})
	An integer valued random process $\{\mathcal{N}(\textbf{t})\}_{\textbf{t}\in\mathbb{R}^d_+}$ is a multiparameter counting process if
	\begin{enumerate}
		\item $\mathcal{N}(\textbf{0})=0$,
		\item $\mathcal{N}(\textbf{t})\ge 0$ for all $\textbf{t}\in\mathbb{R}^d_+$,
		\item $\mathcal{N}(\textbf{s})\le \mathcal{N}(\textbf{t})$ for each $\mathbf{s}\preceq \mathbf{t}$.
	\end{enumerate} 
\end{definition}
\begin{definition}(\textbf{Multiparameter GCP})
	Let $\{\mathscr{M}(\textbf{t})\}_{\textbf{t}\in\mathbb{R}^d_+}$ be a multiparameter counting process. Then, the process $\{\mathscr{M}(\textbf{t})\}_{\textbf{t}\in\mathbb{R}^d_+}$ is a multiparameter GCP with transition rates $\mathbf{\Lambda}_j=(\lambda_{j1}, \lambda_{j2}, \dots, \lambda_{jd})\succ \textbf{0}$, $j=1,2,\dots,k$ if the following holds:\\
	\noindent(i) $\mathscr{M}(\mathbf{0})=0$,\\
	\noindent(ii) it has independent increments, that is, for each $\mathbf{t}^{(i)}\in\mathbb{R}^d_+$ such that $\mathbf{0}=\mathbf{t}^{(0)}\prec\mathbf{t}^{(1)}\prec\dots\prec\mathbf{t}^{(n)}$, the increments $\mathscr{M}(\mathbf{t}^{(1)})-\mathscr{M}(\mathbf{t}^{(0)})$, $\mathscr{M}(\mathbf{t}^{(2)})-\mathscr{M}(\mathbf{t}^{(1)})$, $\dots$, $\mathscr{M}(\mathbf{t}^{(n)})-\mathscr{M}(\mathbf{t}^{(n-1)})$ are mutually independent,\\
	\noindent(iii) it has stationary increments, that is, for $\mathbf{s}\in\mathbb{R}^d_+$ and $\mathbf{t}\in\mathbb{R}^d_+$ such that $\mathbf{s}\preceq \mathbf{t}$, the random variable $\mathscr{M}(\mathbf{t})-\mathscr{M}(\mathbf{s})$ and $\mathscr{M}(\mathbf{t}-\mathbf{s})$ are equal in distribution, that is, $\mathscr{M}(\mathbf{t})-\mathscr{M}(\mathbf{s})\overset{d}{=}\mathscr{M}(\mathbf{t}-\mathbf{s})$,\\
	\noindent(iv) its one dimensional distribution is given by \eqref{MGCPMF}, that is, $\mathrm{Pr}(\mathscr{M}(\textbf{t})=n)=p(n,\mathbf{t})$, $n\ge 0$. 
\end{definition}
The pgf of multiparameter GCP is given by
\begin{equation}\label{MPGPGF}
	\mathbb{E}(u^{\mathscr{M}(\mathbf{t})})=\exp\Big(-\sum_{j=1}^{k}\mathbf{\Lambda}_j
	\cdot \mathbf{t}(1-u^j)\Big), \ |u|\le 1, \ \textbf{t}\in\mathbb{R}^d_+.
\end{equation}
It has the following moment generating function (mgf):
\begin{equation}\label{MPgcpmgf}
	\mathbb{E}(e^{u\mathscr{M}(\textbf{t})})
	=\exp\Big(-\sum_{j=1}^{k}\mathbf{\Lambda}_j\cdot\mathbf{t}(1-e^{uj})\Big), \ u\in\mathbb{R},\ \textbf{t}\in\mathbb{R}^d_+. 
\end{equation}
Its mean and variance are $\mathbb{E}(\mathscr{M}(\mathbf{t}))=\sum_{j=1}^{k}j\mathbf{\Lambda}_j
\cdot \mathbf{t}$ and $\operatorname{Var}(\mathscr{M}(\mathbf{t}))=\sum_{j=1}^{k}j^2\mathbf{\Lambda}_j
\cdot \mathbf{t}$, respectively.

The following result establishes the existence of a multiparameter GCP.
\begin{theorem}\label{MPexis}
	Let $\{\mathscr{N}_1(\mathbf{t})\}_{\mathbf{t}\in\mathbb{R}_+^d}$, $\{\mathscr{N}_2(\mathbf{t})\}_{\mathbf{t}\in\mathbb{R}_+^d}$, $\dots$, $\{\mathscr{N}_k(\mathbf{t})\}_{\mathbf{t}\in\mathbb{R}_+^d}$ be independent multiparameter counting processes. Then, the weighted process
	\begin{equation*}
		\mathscr{M}(\mathbf{t})=\mathscr{N}_1(\mathbf{t})+2\mathscr{N}_2(\mathbf{t})+\dots+l\mathscr{N}_l(\mathbf{t})
	\end{equation*} 
	is a multiparameter GCP with transition parameters $\mathbf{\Lambda}_j=(\lambda_{j1},\lambda_{j2},\dots,\lambda_{jd})\succ\textbf{0}$, $j=1,2,\dots,l$ for all $1\le l\le k$ iff  each $\{\mathscr{N}_j(\mathbf{t})\}_{\mathbf{t}\in\mathbb{R}_+^d}$, $j=1,2,\dots,k$ is a multiparameter Poisson process with transition parameter $\mathbf{\Lambda}_j$.    
\end{theorem}
\begin{proof}
	First we prove the if part. As $\mathscr{N}_j(\mathbf{0})=0$ for all $j=1,2,\dots,k$, we have $\mathscr{M}(\mathbf{0})=\sum_{j=1}^{l}j\mathscr{N}_j(\mathbf{0})=0$ for all $1\le l\le k$. The stationary increments property of $\{\mathscr{M}(\mathbf{t})\}_{\mathbf{t}\in\mathbb{R}_+^d}$ follows on using the stationary increments property of multiparameter Poisson processes. Also, the independence of $\{\mathscr{N}_j(\mathbf{t})\}_{\mathbf{t}\in\mathbb{R}_+^d}$'s and their independent increments property can be used to show that $\{\mathscr{M}(\mathbf{t})\}_{\mathbf{t}\in\mathbb{R}_+^d}$ has independent increments. Moreover, the pgf of $\{\mathscr{M}(\mathbf{t})\}_{\mathbf{t}\in\mathbb{R}_+^d}$ is given by
	\begin{align}
		\mathbb{E}(u^{\sum_{j=1}^{l}j\mathscr{N}_j(\mathbf{t})})&=\prod_{j=1}^{l}\mathbb{E}(u^{j\mathscr{N}_j(\mathbf{t})}),\ \text{(using independence of $\mathscr{N}_j(\mathbf{t})$'s)}\nonumber\\
		&=\exp\Big(\sum_{j=1}^{l}\mathbf{\Lambda}_j \cdot \mathbf{t} (u^j-1)\Big),\ |u|\le 1,\ \textbf{t}\in\mathbb{R}^d_+,\label{pgfrex}
	\end{align}
	where the last step follows on using Eq. (3.3) of \cite{Vishwakarma2025}. So, \eqref{pgfrex} is the pgf of a multiparameter GCP. Thus, the weighted process $\{\sum_{j=1}^{l}j\mathscr{N}_j(\mathbf{t})\}_{\mathbf{t}\in\mathbb{R}_+^d}$ is a multiparameter GCP for all $1\le l\le k$ with transition parameters $\mathbf{\Lambda}_{j}$'s.
	
	To establish the only if part, we use the method of induction as follows:
	For $k=1$, the result holds true, that is, $\{\mathscr{N}_1(\mathbf{t})\}_{\mathbf{t}\in\mathbb{R}_+^d}$ is a multiparameter Poisson process with transition parameter $\mathbf{\Lambda}_1\succ \textbf{0}$. Let us assume that the result holds true for $k=m-1\ge 2$. Now, by using the independence of $\mathscr{N}_j(\mathbf{t})$'s for $k=m>2$, we get
	\begin{equation}\label{m}
		\mathbb{E}(e^{u\sum_{j=1}^{m}j\mathscr{N}_j(\mathbf{t})})  =\mathbb{E}(e^{umN_{m}(\mathbf{t})})\prod_{j=1}^{m-1}e^{\mathbf{\Lambda}_j\cdot\mathbf{t}(e^{uj}-1)},\ u\in\mathbb{R}
	\end{equation}
	which follows on using the induction hypothesis, that is, $\sum_{j=1}^{m-1}j\mathscr{N}_j(\mathbf{t})$ is a multiparameter GCP and $\{\mathscr{N}_1(\mathbf{t})\}_{\mathbf{t}\in\mathbb{R}_+^d}$, $\{\mathscr{N}_2(\mathbf{t})\}_{\mathbf{t}\in\mathbb{R}_+^d}$, $\dots$, $\{\mathscr{N}_{m-1}(\mathbf{t})\}_{\mathbf{t}\in\mathbb{R}_+^d}$ are multiparameter Poisson processes.
	Also, by using \eqref{MPgcpmgf}, we have
	\begin{equation}\label{mm}
		\mathbb{E}(e^{u\sum_{j=1}^{m}j\mathscr{N}_j(\mathbf{t})})
		=\exp\bigg(\sum_{j=1}^{m}\mathbf{\Lambda}_j\cdot\mathbf{t}(e^{uj}-1)\bigg).
	\end{equation}
	From \eqref{m} and \eqref{mm}, we get
	\begin{equation}\label{a}
		\mathbb{E}(e^{um\mathscr{N}_{m}(\mathbf{t})})=e^{\mathbf{\Lambda}_{m}\cdot\mathbf{t}(e^{um}-1)}.
	\end{equation} 
	Thus, $\{\mathscr{N}_{m}(\mathbf{t})\}_{\mathbf{t}\in\mathbb{R}_+^d}$ has a multiparameter Poisson distribution with parameter $\mathbf{\Lambda}_{m}\cdot\textbf{t}$.
	
	Let $\mathbf{s}\in\mathbb{R}^d_+$ and $\mathbf{t}\in\mathbb{R}^d_+$ be such that $\mathbf{0}\prec\mathbf{s}\preceq\mathbf{t}$. Then, by using the independence of $N_j(\mathbf{t})$'s for $k=m$, we get
	\begin{align}
		\mathbb{E}\bigg(\exp\bigg(u\sum_{j=1}^{m}j(\mathscr{N}_j(\mathbf{t})&-\mathscr{N}_j(\mathbf{s}))\bigg)\bigg)\nonumber
		\\&=\mathbb{E}(e^{um(\mathscr{N}_m(\mathbf{t})-\mathscr{N}_m(\mathbf{s}))})\prod_{j=1}^{m-1}\mathbb{E}(e^{uj(\mathscr{N}_j(\mathbf{t})-\mathscr{N}_j(\mathbf{s}))})\nonumber\\
		&=\mathbb{E}(e^{um(\mathscr{N}_m(\mathbf{t})-\mathscr{N}_m(\mathbf{s}))})\prod_{j=1}^{m-1}e^{\mathbf{\Lambda}_j\cdot(\mathbf{t}-\mathbf{s})(e^{uj}-1)},\label{s}
	\end{align}
	where the last step is obtained on using the stationary increments property of  $\{\mathscr{N}_j(\mathbf{t})\}_{\mathbf{t}\in\mathbb{R}^d_+}$, $1\le j\le m-1$. Also, we have 
	\begin{equation}\label{ss}
		\mathbb{E}\bigg(\exp\bigg(u\sum_{j=1}^{m}j(\mathscr{N}_j(\mathbf{t})-\mathscr{N}_j(\mathbf{s}))\bigg)\bigg)=\prod_{j=1}^{m}e^{\mathbf{\Lambda}_j\cdot(\mathbf{t}-\mathbf{s})(e^{uj}-1)}
	\end{equation}
	which follows on using the stationary increments of multiparameter GCP.
	On equating the right hand sides of \eqref{s} and \eqref{ss}, we get
	\begin{equation*}
		\mathbb{E}(e^{um(\mathscr{N}_m(\mathbf{t})-\mathscr{N}_m(\mathbf{s}))})=e^{\mathbf{\Lambda_m}\cdot(\mathbf{t}-\mathbf{s})(e^{um}-1)}.
	\end{equation*} 
	Thus, $\{\mathscr{N}_m(\mathbf{t})\}_{\mathbf{t}\in\mathbb{R}^d_+}$ has the stationary increments property, that is, $\mathscr{N}_m(\mathbf{t})-\mathscr{N}_m(\mathbf{s})\overset{d}{=}\mathscr{N}_m(\mathbf{t}-\mathbf{s})$, which follows from \eqref{a}.
	
	Let $\mathbf{t}^{(i)}\in\mathbb{R}^d_+$, $i=1,2,\dots,r$ and $\textbf{0}\preceq \mathbf{t}^{(0)}\prec\mathbf{t}^{(1)}\prec\dots\prec\mathbf{t}^{(r)}$. By using the independence of $N_j(\mathbf{t})$'s for $k=m$, we get
	\begin{align}
		\mathbb{E}&\bigg(\exp\bigg(\sum_{i=1}^{r}u_i\sum_{j=1}^{m}j(\mathscr{N}_j(\mathbf{t}^{(i)})-\mathscr{N}_j(\mathbf{t}^{(i-1)})\bigg)\bigg)\nonumber\\
		&=\prod_{i=1}^{r}\mathbb{E}\Big(e^{u_im(\mathscr{N}_m(\mathbf{t}^{(i)})-\mathscr{N}_m(\mathbf{t}^{(i-1)})}\Big)\mathbb{E}\Big(\exp\Big(u_i\sum_{j=1}^{m-1}j(\mathscr{N}_j(\mathbf{t}^{(i)})-\mathscr{N}_j(\mathbf{t}^{(i-1)}))\Big)\Big)
		\nonumber\\
		&=\prod_{i=1}^{r}\mathbb{E}\bigg(e^{u_im(\mathscr{N}_m(\mathbf{t}^{(i)})-\mathscr{N}_m(\mathbf{t}^{(i-1)}))}\bigg)\prod_{j=1}^{m-1}e^{\mathbf{\Lambda}_j\cdot(\mathbf{t}^{(i)}-\mathbf{t}^{(i-1)})(e^{u_ij}-1)},\label{i}
	\end{align}
	where we obtain the last step using stationary increments of $\mathscr{N}_j(\mathbf{t})$'s. 
	Now, by using the independent increments of multiparameter GCP $\{\sum_{j=1}^{m}j\mathscr{N}_j(\mathbf{t})\}_{\mathbf{t}\in\mathbb{R}^d_+}$, we obtain 
	\begin{equation}\label{ii}
		\mathbb{E}\bigg(\exp\bigg(\sum_{i=1}^{r}u_i\sum_{j=1}^{m}j(\mathscr{N}_j(\mathbf{t}^{(i)})-\mathscr{N}_j(\mathbf{t}^{(i-1)}))\bigg)\bigg)=\prod_{i=1}^{r}\prod_{j=1}^{m}e^{\mathbf{\Lambda}_j\cdot(\mathbf{t}^{(i)}-\mathbf{t}^{(i-1)})(e^{u_ij}-1)}.
	\end{equation}
	From \eqref{i} and \eqref{ii}, we get
	\begin{equation*}
		\mathbb{E}\bigg(e^{\sum_{i=1}^{r}u_im(\mathscr{N}_m(\mathbf{t}^{(i)})-\mathscr{N}_m(\mathbf{t}^{(i-1)}))}\bigg)=\prod_{i=1}^{r}\mathbb{E}(e^{u_im(\mathscr{N}_m(\mathbf{t}^{(i)})-\mathscr{N}_m(\mathbf{t}^{(i-1)}))}).
	\end{equation*}
	So, $\{\mathscr{N}_m(\mathbf{t})\}_{\mathbf{t}\in\mathbb{R}^d_+}$ has independent increments. Thus, $\{\mathscr{N}_m(\mathbf{t})\}_{\mathbf{t}\in\mathbb{R}^d_+}$ is a multiparameter Poisson process with transition parameter $\mathbf{\Lambda}_m\succ\textbf{0}$. This completes the proof.  
\end{proof}
\begin{remark}\label{summpp}
	Let $\{\mathscr{M}(\textbf{t})\}_{\textbf{t}\in\mathbb{R}^d_+}$ be a multiparameter GCP with transition parameters  $\boldsymbol{\Lambda}_j=(\lambda_{j1},\lambda_{j2},\dots,\lambda_{jd})\succ\textbf{0}$, $j=1,2,\dots,k$. Then, there exist independent multiparameter Poisson processes $\{\mathscr{N}_1(\textbf{t})\}_{\textbf{t}\in\mathbb{R}^d_+}$, $\{\mathscr{N}_2(\textbf{t})\}_{\textbf{t}\in\mathbb{R}^d_+}$, $\dots$, $\{\mathscr{N}_k(\textbf{t})\}_{\textbf{t}\in\mathbb{R}^d_+}$  with transition rates $\boldsymbol{\Lambda}_1$, $\boldsymbol{\Lambda}_2$, $\dots$, $\boldsymbol{\Lambda}_k$, respectively, such that
	$\mathscr{M}(\textbf{t})\overset{d}{=}\sum_{j=1}^{k}j\mathscr{N}_j(\textbf{t})$ for all $\textbf{t}=(t_1,t_2,\dots,t_d)\in\mathbb{R}^d_+
	$. So, the multiparameter GCP is equal in distribution to an additive L\'evy process in the sense of  \cite{Khoshnevisan2002a}.   
	We refer the reader to  \cite{Khoshnevisana2003} for more details on additive L\'evy processes.  
\end{remark}

Alternatively, the existence of a multiparameter GCP can be established using independent GCPs as follows:
\begin{theorem}
	Let $\{M_1(t_1)\}_{t_1\ge0}$, $\{M_2(t_2)\}_{t_2\ge0}$, $\dots$, $\{M_d(t_d)\}_{t_d\ge0}$ be independent counting processes. Then, the multiparameter counting process $\{\mathscr{M}(\textbf{t})\}_{\textbf{t}\in\mathbb{R}^d_+}$ defined as
	\begin{equation}\label{MPgexc}
		\mathscr{M}(\textbf{t})\coloneqq M_1(t_1)+M_2(t_2)+\dots+M_d(t_d),\ \textbf{t}\in\mathbb{R}^d_+,
	\end{equation}
	is a multiparameter GCP with transition rates $\boldsymbol{\Lambda}_j=(\lambda_{j1},\lambda_{j2},\dots,\lambda_{jd})\succ \textbf{0}$, $j=1,2,\dots,k$ if and only if for each $i=1,2,\dots,d$, $\{M_i(t_i)\}_{t_i\ge0}$ is a GCP which performs independently $k$ kinds of jumps with positive rates $\lambda_{ji}>0$, $j=1,2,\dots,k$.
\end{theorem}
\begin{proof}
	The proof follows similar lines to that of Theorem 3.1 of \cite{Vishwakarma2025}. Hence, it is omitted.
\end{proof}
\begin{remark}\label{MGasGcp}
	In view of \cite{Khoshnevisan2002a}, the multiparameter GCP as defined in \eqref{MPgexc} is an additive L\'evy process. Also, for any multiparameter GCP $\{\mathscr{M}(\textbf{t})\}_{\textbf{t}\in\mathbb{R}^d_+}$ with transition rates $\boldsymbol{\Lambda}_j$, $j=1,2,\dots,k$, there exist independent GCPs $\{M_i(t_i)\}_{t_i\ge0}$, $i=1,2,\dots,d$ with transition rates $\lambda_{ji}$, $j=1,2,\dots,k$ such that $\mathscr{M}(\textbf{t})\overset{d}{=}\sum_{i=1}^{d}M_i(t_i)$ for all $\textbf{t}=(t_1,t_2,\dots,t_d)\in\mathbb{R}^d_+$.
\end{remark}
Alternatively,  the pgf given in \eqref{MPGPGF} can be obtained by using Remark \ref{MGasGcp} as follows:
\begin{align*}
	\mathbb{E}(u^{\mathscr{M}(\mathbf{t})})&=\prod_{i=1}^{d}\mathbb{E}(u^{M_i(t_i)}),\ \text{(as $M_i(t_i)$'s are independent)}\\
	&=\prod_{i=1}^{d}\exp\Big(-\sum_{j=1}^{k}\lambda_{ji}t_i(1-u^j)\Big)\\
	&=\exp\Big(-\sum_{j=1}^{k}\mathbf{\Lambda}_j
	\cdot \mathbf{t}(1-u^j)\Big), \ |u|\le 1, \ \textbf{t}\in\mathbb{R}^d_+.
\end{align*}

Also, by using Remark \ref{MGasGcp}, the distribution function $p(n,\textbf{t})=\mathrm{Pr}\{\mathscr{M}(\textbf{t})=n\}$, $n\ge0$ can be obtained in the following form: 
\begin{equation}\label{MGCpmf}
	p(n,\textbf{t})=\sum_{\Theta(n,d)}\Big(\prod_{i=1}^{d}\sum_{\Omega(k,n_i)}\prod_{j=1}^{k}\frac{(\lambda_{ji}t_i)^{n_{ji}}}{n_{ji}!}e^{-\lambda_{ji}t_i}\Big),
\end{equation}
where we have used the distribution of GCP (see \cite[Eq. (2.9)]{Di Crescenzo2016}). 
Here, $\Theta(n,d)=\{(n_1,n_2,\dots,n_d):\sum_{i=1}^{d}n_i=n,\  n_i\in\mathbb{N}_0\}$ and $\Omega(k,n_i)=\{(n_{1i},n_{2i},\dots, n_{ki}):\sum_{j=1}^{k}n_{ji}=n_i,\ n_{ji}\in\mathbb{N}_0\}$. 
It is important to note that \eqref{MGCPMF} and \eqref{MGCpmf}
are equivalent representation of the distribution of $\{\mathscr{M}(\textbf{t})\}_{\textbf{t}\in\mathbb{R}^d_+}$.

\subsection{Integrals of multiparameter GCP}
The Riemann-Liouville (RL) fractional integral of homogeneous Poisson process and its time-changed variant was studied by Orsingher and Polito \cite{Orsingher2013}. In  \cite{Kataria2024}, a similar study is done for the fractional Poisson random field over a rectangle. For more details on integrated random processes and their potential applications, we refer the reader to \cite{Khandakar2025}, \cite{Vishwakarma2024a}, \cite{Vishwakarma2024b}, and references therein. 

Here, we discuss about the fractional integrals of multiparameter GCP over a rectangle in $\mathbb{R}^d_+$.
For $\boldsymbol{\alpha}=(\alpha_1,\alpha_2,\dots,\alpha_d)\succ \textbf{0}$, we define the RL fractional integral of multiparameter GCP as follows:
{\small\begin{equation*}
		\mathscr{X}^{\boldsymbol{\alpha}}(\mathbf{t})=\frac{1}{\Gamma(\alpha_1)\dots\Gamma(\alpha_d)}\int_{0}^{t_1}\dots\int_{0}^{t_d}(t_1-s_1)^{\alpha_1-1}\dots(t_d-s_d)^{\alpha_d-1}\mathscr{M}(\mathbf{s})\,\mathrm{d}s_1\dots\mathrm{d}s_d,\ \mathbf{t}\in\mathbb{R}^d_+.
\end{equation*}}
\begin{remark}
	For $\boldsymbol{\alpha}=\textbf{1}$, $\mathscr{X}^{\boldsymbol{\alpha}}(\textbf{t})$ reduces to the Riemann integral of multiparameter GCP. We denote it by $\mathscr{X}(\textbf{t})$.
\end{remark}
Its mean is given by
{\small\begin{align*}
		&\mathbb{E}(\mathscr{X}^{\boldsymbol{\alpha}}(\mathbf{t}))\\
		&=\frac{1}{\Gamma(\alpha_1)\dots\Gamma(\alpha_d)}\sum_{j=1}^{k}\sum_{l=1}^{d}j\lambda_{jl}\int_{0}^{t_1}\dots\int_{0}^{t_d}(t_1-s_1)^{\alpha_1-1}\dots(t_d-s_d)^{\alpha_d-1}s_l\,\mathrm{d}s_1\dots\mathrm{d}s_d\nonumber\\
		&=\sum_{l=1}^{d}\sum_{j=1}^{k}j\lambda_{jl}\frac{1}{\Gamma(\alpha_1)\dots\Gamma(\alpha_d)}\Big(\prod_{\substack{i=1\\i\ne l}}^{d}\frac{t_i^{\alpha_i}}{\alpha_i}\Big)t_l^{\alpha_l+1}\frac{\Gamma(\alpha_l)}{\Gamma(\alpha_l+2)}\nonumber\\
		&=\sum_{l=1}^{d}\sum_{j=1}^{k}j\lambda_{jl}\Big(\prod_{\substack{i=1\\i\ne l}}^{d}\frac{t_i^{\alpha_i}}{\Gamma(\alpha_i+1)}\Big)\frac{t_l^{\alpha_l+1}}{\Gamma(\alpha_l+2)}.
	\end{align*}
}
From Remark \ref{summpp} and Remark 3.3 of \cite{Vishwakarma2025}, we have 
{\small\begin{align*}
		\mathscr{X}^{\boldsymbol{\alpha}}(\mathbf{t})&\overset{d}{=}\sum_{j=1}^{k}\frac{j}{\Gamma(\alpha_1)\dots\Gamma(\alpha_d)}\int_{0}^{t_1}(t_1-s_1)^{\alpha_1-1}\dots\int_{0}^{t_d}(t_d-s_d)^{\alpha_d-1}N_j(\textbf{s})\,\mathrm{d}s_1\dots\mathrm{d}s_d\\
		&\overset{d}{=}\sum_{j=1}^{k}\sum_{i=1}^{d}j\Big(\prod_{\substack{r=1\\r\ne i}}^{d}\frac{t_r^{\alpha_r}}{\Gamma(\alpha_r+1)}\Big)\frac{1}{\Gamma(\alpha_i)}\int_{0}^{t_i}(t_i-s_i)^{\alpha_i-1}N_{ji}(s_i)\mathrm{d}s_i,
\end{align*}}
where $\{N_{ji}(t_i)\}_{t_i\ge0}$'s are independent Poisson processes with transition rates $\lambda_{ji}>0$, respectively.

By using Eq. (3.19) of  \cite{Vishwakarma2025}, the variance of $\mathscr{X}^{\boldsymbol{\alpha}}(\mathbf{t})$ is obtained in the following form:
\begin{align*}
	\operatorname{Var}(\mathscr{X}^{\boldsymbol{\alpha}}(\mathbf{t}))&=\sum_{j=1}^{k}j^2\sum_{i=1}^{d}\frac{\lambda_{ji}t_i^{2\alpha_i+1}}{(2\alpha_i+1)\Gamma^2(\alpha_i+1)}\Big(\prod_{\substack{r=1\\r\ne i}}^{d}\frac{t_r^{\alpha_r}}{\Gamma(\alpha_r+1)}\Big)^2,\ \textbf{t}\in\mathbb{R}^d_+.
\end{align*}
\begin{remark}
	For $k=1$, the mean and variance of RL fractional integral of multiparameter GCP reduces to that of  RL fractional integral of multiparameter Poisson process (see \cite[Section 3.1]{Vishwakarma2025}).
\end{remark}
The following result gives a compound sum representation of the Riemann integral of multiparameter GCP.
\begin{theorem}
	For $i=1,2,\dots,d$ and $j=1,2,\dots,k$, let $\{N_{ji}(t_i)\}_{t_i\ge0}$'s be independent one parameter Poisson processes with positive rates $\lambda_{ji}$ and $Y_{i1}$, $Y_{i2}$, $\dots$ be independent and identically distributed random variables with uniform distribution on $[0,t_i]$. Then, 
	\begin{equation*}
		\mathscr{X}(\textbf{t})
		\overset{d}{=}\sum_{j=1}^{k}\sum_{i=1}^{d}j\Big(\prod_{\substack{r=1\\r\ne i}}^{d}t_r\Big)\sum_{l=1}^{N_{ji}(t_i)}Y_{il},
	\end{equation*}
	where $ \textbf{t}=(t_1,t_2,\dots,t_d)\in\mathbb{R}^d_+$.
\end{theorem}
\begin{proof}
	By using Remark \ref{summpp}, we have
	\begin{equation*}
		\mathscr{X}(\textbf{t})\overset{d}{=}\sum_{j=1}^{k}j\int_{0}^{t_1}\dots\int_{0}^{t_d}N_{j}(\textbf{s})\,\mathrm{d}s_1\dots\mathrm{d}s_d,
	\end{equation*}
	where $\{N_{j}(\textbf{s})\}_{\textbf{s}\in\mathbb{R}^d_+}$'s are independent multiparameter Poisson processes with transition rates $\boldsymbol{\Lambda}_j\succ \textbf{0}$. By using Theorem 3.2 of \cite{Vishwakarma2025}, the required result follows.
\end{proof}
\begin{remark}
	For $j=1,2,\dots,k$ and $r=1,2,\dots,d$, let $\{N_{jr}(s_r)\}_{s_r\ge 0}$ be independent one parameter Poisson processes with positive rates $\lambda_{jr}$, respectively. Then, the characteristic function of $\mathscr{X}(\mathbf{t})$ is given by
	{\small	\begin{align}
			&\mathbb{E}(e^{i\eta\mathscr{X}(\mathbf{t})})\nonumber\\
			&=\mathbb{E}\Big(\exp\Big(i\eta\sum_{j=1}^{k}\sum_{l=1}^{d}\Big(\prod_{r\ne l}t_r\Big)\int_{0}^{t_l}jN_{jl}(s_l)\mathrm{d}s_l\Big)\Big)\nonumber\\
			&=\prod_{j=1}^{k}\prod_{l=1}^{d}\mathbb{E}\Big(\exp\Big(i\eta j\Big(\prod_{r\ne l}t_r\Big)\int_{0}^{t_l}N_{jl}(s_l)\mathrm{d}s_l\Big)\Big)\nonumber\\
			&=\exp\Big(i\eta\sum_{j=1}^{k}\sum_{l=1}^{d}\frac{ j\lambda_{jl}t_l^2(\prod_{r\ne l}t_r)}{2}-\frac{\eta^2}{2}\sum_{j=1}^{k}\sum_{l=1}^{d}\frac{j^2\lambda_{jl}t_l^3 (\prod_{r\ne l}t_r^2)}{3}+\sum_{j=1}^{k}j\sum_{l=1}^{d}o(t_l^3)\Big),\label{charMPG}
	\end{align}}
	where the last step follows on using Eq. (4.10) of \cite{Orsingher2013}. Here, $o(t_l^3)/t_l^3\to 0$ as $t_l\to 0$. 
	
	Thus, for sufficiently small $\textbf{t}\in\mathbb{R}^d_+$, the Riemann integral of multiparameter GCP has a Gaussian distribution with mean $\sum_{j=1}^{k}\sum_{l=1}^{d}j\lambda_{jl} (\prod_{r\ne l}t_r)t_l^2/2$ and variance $\sum_{j=1}^{k}\sum_{l=1}^{d}j^2\lambda_{jl}$ $(\prod_{r\ne l}t_r^2)t_l^3/3$.
\end{remark}

\begin{remark}
	For $k=1$, the characteristic function of Riemann integral \eqref{charMPG} reduces to that of multiparameter Poisson process (see \cite[Section 3.1]{Vishwakarma2025}).
\end{remark}
\section{Time-changed variants of  multiparameter GCP}\label{sec4}
In this section, we study two time-changed variants of the multiparameter GCP where the time changing components are multiparameter stable subordinator and its inverse.

\subsection{Time-changed with multiparameter stable subordinator}
Let
$\{\mathscr{M}(\textbf{t})\}_{\textbf{t}\in\mathbb{R}^d_+}$ be a multiparameter GCP with transition parameters $\boldsymbol{\Lambda}_j=(\lambda_{j1},\lambda_{j2},\dots,\lambda_{jd})\succ\textbf{0}$, $j=1,2,\dots,k$. Also, let $\boldsymbol{\mathscr{D}}^{\boldsymbol{\alpha}}(\textbf{t})=(D_1^{\alpha_1}(t_1),D_2^{\alpha_2}(t_2)$, $\dots, D_d^{\alpha_d}(t_d))$, $\boldsymbol{\alpha}=(\alpha_1,\alpha_2,\dots,\alpha_d)$, $\alpha_i\in(0,1)$ be a multiparameter stable subordinator as defined in \eqref{mssdef} such that the component processes $\{D_i^{\alpha_i}(t_i)\}_{t_i\ge 0}$'s are independent stable subordinators. 

Let us consider the following time-changed variant of multiparameter GCP:
\begin{equation}\label{alphmgcp}
	\mathscr{M}^{\boldsymbol{\alpha}}(\textbf{t})\coloneqq\mathscr{M}(\boldsymbol{\mathscr{D}}^{\boldsymbol{\alpha}}(\textbf{t})),\ \textbf{t}\in\mathbb{R}^d_+,
\end{equation}
where  $\{\mathscr{M}(\textbf{t})\}_{\textbf{t}\in\mathbb{R}^d_+}$ is independent of  $\{\boldsymbol{\mathscr{D}}^{\boldsymbol{\alpha}}(\textbf{t})\}_{\textbf{t}\in\mathbb{R}^d_+}$. 
Its pgf $G^{\boldsymbol{\alpha}}(u,\textbf{t})=\mathbb{E}(u^{\mathscr{M}^{\boldsymbol{\alpha}}(\textbf{t})})$, $|u|\le 1$ can be obtained on using the law of iterated expectations as follows:
{\small\begin{align}\label{pgfalpMGCP}
		G^{\boldsymbol{\alpha}}(u,\textbf{t})&=\mathbb{E}(\mathbb{E}(u^{\mathscr{M}^{\boldsymbol{\alpha}}(\textbf{t})}|\boldsymbol{\mathscr{D}}^{\boldsymbol{\alpha}}(\textbf{t})))\nonumber\\
		&=\mathbb{E}\Big(\exp\Big(-\sum_{j=1}^{k}\boldsymbol{\Lambda}_j\cdot\boldsymbol{\mathscr{D}}^{\boldsymbol{\alpha}}(\textbf{t})(1-u^j)\Big)\Big),\ \text{(using \eqref{MPGPGF})}\nonumber\\
		&=\prod_{i=1}^{d}\mathbb{E}\Big(\exp\Big(-\sum_{j=1}^{k}\lambda_{ji}D_i^{\alpha_i}(t_i)(1-u^j)\Big)\Big),\ \text{(as $D_i^{\alpha_{i}}(t_i)$'s are independent)}\nonumber\\
		&=\exp\Big(-\sum_{i=1}^{d}t_i\Big(\sum_{j=1}^{k}\lambda_{ji}(1-u^j)\Big)^{\alpha_i}\Big),
\end{align}}
where the last step follows from the Laplace transform of stable subordinator (see Section \ref{alphastab}). Thus, $\mathbb{E}(\mathscr{M}^{\boldsymbol{\alpha}}(\textbf{t}))^m=\infty$ for all $m\ge1$. For any $i=1,2,\dots,d$, it can be shown that the pgf of $\{\mathscr{M}^{\boldsymbol{\alpha}}(\textbf{t})\}_{\textbf{t}\in\mathbb{R}^d_+}$ solves the following differential equation:
\begin{equation}\label{pgfmgdeq}
	\frac{\partial}{\partial t_i}G^{\boldsymbol{\alpha}}(u,\textbf{t})=-\Big(\sum_{j=1}^{k}\lambda_{ji}(1-u^j)\Big)^{\alpha_i}G^{\boldsymbol{\alpha}}(u,\textbf{t}), \ |u|\le1,
\end{equation}
with $G^{\boldsymbol{\alpha}}(u,\textbf{0})=1$. 

\begin{remark}\label{stabsumm}
	Let  $\boldsymbol{\mathscr{D}}^{\boldsymbol{\alpha}}(\textbf{t})=(D_1^{\alpha_1}(t_1),D_2^{\alpha_2}(t_2),\dots,D_d^{\alpha_d}(t_d))$ be a multiparameter stable subordinator such that  $\{D_i^{\alpha_i}(t_i)\}_{t_i\ge 0}$'s are independent stable subordinators. From Remark \ref{MGasGcp}, it follows that there exist independent GCPs $\{M_1(t_1)\}_{t_1\ge0}$, $\{M_2(t_2)\}_{t_2\ge0}$, $\dots$, $\{M_d(t_d)\}_{t_d\ge0}$ with transition parameters $\Lambda_{j1}$, $\Lambda_{j2}$, $\dots$, $\Lambda_{jd}$, $j=1,2,\dots,k$, respectively, such that
	\begin{equation}\label{stabrekk}
		\mathscr{M}^{\boldsymbol{\alpha}}(\textbf{t})\overset{d}{=}M_1(D_1^{\alpha_1}(t_1))+M_2(D_2^{\alpha_2}(t_2))+\dots+M_d(D_d^{\alpha_d}(t_d)),\, \textbf{t}\in\mathbb{R}^d_+,
	\end{equation} 
	where $\{M_i(t_i)\}_{t_i\ge0}$'s are independent of $\{D_i^{\alpha_i}(t_i)\}_{t_i\ge0}$'s.
	So, the component processes $\{M_i(D_i^{\alpha_i}(t_i))\}_{t_i\ge0}$'s  are independent generalized space fractional counting processes (see \cite[Section 4.1]{Kataria2022}).
	Alternatively, the equality in distribution in \eqref{stabrekk} can be established by evaluating the pgf of $\sum_{i=1}^{d}M_i(D_i^{\alpha_i}(t_i))$, which coincides with \eqref{pgfalpMGCP}.
\end{remark}
Next, we use the result given in \eqref{stabrekk} to obtain the distribution of $\{\mathscr{M}^{\boldsymbol{\alpha}}(\textbf{t})\}_{\textbf{t}\in\mathbb{R}^d_+}$, and to derive its governing system of differential equations.
\begin{theorem}\label{thmalph}
	The distribution $p^{\boldsymbol{\alpha}}(n,\textbf{t})=\mathrm{Pr}\{\mathscr{M}^{\boldsymbol{\alpha}}(\textbf{t})=n\}$, $n\ge0$ is given by 
	{\small	\begin{equation*}
			p^{\boldsymbol{\alpha}}(n,\textbf{t})=\sum_{\Theta(n,d)}\prod_{i=1}^{d}\sum_{\Omega(k,n_i)}\Big(\prod_{j=1}^{k}\Big(-\frac{\lambda_{ji}}{\mu_i}\Big)^{n_{ji}}\frac{1}{n_{ji}!}\Big)\sum_{r_i=0}^{\infty}\frac{(-\mu_{i}^{\alpha_i}t_i)^{r_i}\Gamma(\alpha_ir_i+1)}{r_i!\Gamma(\alpha_ir_i+1-\sum_{j=1}^{k}n_{ji})},
	\end{equation*}}
	where $\mu_i=\lambda_{1i}+\lambda_{2i}+\dots+\lambda_{ki}$, $\Theta (n,d)=\{(n_{1},n_{2}$, $\dots,n_{d}):\sum_{i=1}^{d}n_{i}=n,\, n_{i}\in\mathbb{N}_0\}$ and $\Omega(k,n_i)=\{(n_{1i},n_{2i},\dots,n_{ki}):\sum_{j=1}^{k}jn_{ji}=n_i,\, n_{ji}\in\mathbb{N}_0\}$. Also, for any $i=1,2,\dots,d$, it solves the following system of differential equations:
	\begin{equation}\label{defalstmgp}
		\frac{\partial}{\partial t_i}p^{\boldsymbol{\alpha}}(n,\textbf{t})=-\Big(\sum_{j=1}^{k}\lambda_{ji}(I-B^j)\Big)^{\alpha_i}p^{\boldsymbol{\alpha}}(n,\textbf{t}),\ n\ge0,
	\end{equation} 
	with initial conditions $p^{\boldsymbol{\alpha}}(0,\textbf{0})=1$ and $p^{\boldsymbol{\alpha}}(n,\textbf{0})=0$, $n\ge 1$. Here, $B$ is the backward shift operator, that is, $Bp^{\boldsymbol{\alpha}}(n,\textbf{t})=p^{\boldsymbol{\alpha}}(n-1,\textbf{t})$.
\end{theorem}
\begin{proof}
	From \eqref{stabrekk}, we have
	\begin{equation*}
		p^{\boldsymbol{\alpha}}(n,\textbf{t})=\sum_{\Theta(n,d)}\prod_{i=1}^{d}\mathrm{Pr}\{M_i(D_i^{\alpha_i}(t_i))=n_i\}.
	\end{equation*}
	By using the distribution of generalized space fractional counting process (see \cite[Eq. (49)]{Kataria2022}), we obtain the required distribution.
	
	For $i=1,2,\dots,d$, let $\mu_i=\lambda_{1i}+\lambda_{2i}+\dots+\lambda_{ki}$. Then, from \eqref{pgfmgdeq}, we have
	{\small	\begin{align}
			&\sum_{n=0}^{\infty}u^n\frac{\partial}{\partial t_i}p^{\boldsymbol{\alpha}}(n,\textbf{t})\nonumber\\
			&=-\mu_i^{\alpha_i}\Big(1-\frac{1}{\mu_i}\sum_{j=1}^{k}\lambda_{ji}u^j\Big)^{\alpha_i}\sum_{n=0}^{\infty}u^np^{\boldsymbol{\alpha}}(n,\textbf{t})\nonumber\\
			&=-\mu_i^{\alpha_i}\sum_{r_i=0}^{\infty}\binom{\alpha_i}{r_i}\Big(-\frac{1}{\mu_i}\Big)^{r_i}\Big(\sum_{j=1}^{k}\lambda_{ji}u^j\Big)^{r_i}\sum_{n=0}^{\infty}u^np^{\boldsymbol{\alpha}}(n,\textbf{t})\nonumber\\
			&=-\mu_i^{\alpha_i}\sum_{r_i=0}^{\infty}\binom{\alpha_i}{r_i}\Big(-\frac{1}{\mu_i}\Big)^{r_i}\sum_{r_{1i}+r_{2i}+\dots+r_{ki}=r_i}r_i!\prod_{j=1}^{k}\frac{(\lambda_{ji}u^j)^{r_{ji}}}{r_{ji}!}\sum_{n=0}^{\infty}u^np^{\boldsymbol{\alpha}}(n,\textbf{t})\nonumber\\
			&=-\mu_i^{\alpha_i}\sum_{r_i=0}^{\infty}\binom{\alpha_i}{r_i}\Big(-\frac{1}{\mu_i}\Big)^{r_i}\sum_{r_{1i}+r_{2i}+\dots+r_{ki}=r_i}r_i!\prod_{j=1}^{k}\frac{\lambda_{ji}^{r_{ji}}}{r_{ji}!}\sum_{n=0}^{\infty}u^{n+\sum_{j=1}^{k}jr_{ji}}p^{\boldsymbol{\alpha}}(n,\textbf{t})\nonumber\\
			&=-\mu_i^{\alpha_i}\sum_{r_i=0}^{\infty}\binom{\alpha_i}{r_i}\Big(-\frac{1}{\mu_i}\Big)^{r_i}\sum_{r_{1i}+r_{2i}+\dots+r_{ki}=r_i}r_i!\prod_{j=1}^{k}\frac{\lambda_{ji}^{r_{ji}}}{r_{ji}!}\sum_{n=\sum_{j=1}^{k}jr_{ji}}^{\infty}u^{n}p^{\boldsymbol{\alpha}}(n-\sum_{j=1}^{k}jr_{ji},\textbf{t})\nonumber\\
			&=-\mu_i^{\alpha_i}\sum_{r_i=0}^{\infty}\binom{\alpha_i}{r_i}\Big(-\frac{1}{\mu_i}\Big)^{r_i}\sum_{r_{1i}+r_{2i}+\dots+r_{ki}=r_i}r_i!\prod_{j=1}^{k}\frac{\lambda_{ji}^{r_{ji}}}{r_{ji}!}\sum_{n=0}^{\infty}u^{n}p^{\boldsymbol{\alpha}}(n-\sum_{j=1}^{k}jr_{ji},\textbf{t})\label{penstep}\\
			&=-\mu_i^{\alpha_i}\sum_{r_i=0}^{\infty}\binom{\alpha_i}{r_i}\Big(-\frac{1}{\mu_i}\Big)^{r_i}\sum_{r_{1i}+r_{2i}+\dots+r_{ki}=r_i}r_i!\prod_{j=1}^{k}\frac{(\lambda_{ji}B^j)^{r_{ji}}}{r_{ji}!}\sum_{n=0}^{\infty}u^{n}p^{\boldsymbol{\alpha}}(n,\textbf{t})\nonumber\\
			&=-\Big(\sum_{j=1}^{k}\lambda_{ji}(I-B^j)\Big)^{\alpha_i}\sum_{n=0}^{\infty}u^{n}p^{\boldsymbol{\alpha}}(n,\textbf{t}),\label{mmmn}
	\end{align}}
	where \eqref{penstep} follows on using the fact that $p^{\boldsymbol{\alpha}}(n-m,\textbf{t})=0$ for all $n<m$. On comparing the coefficient of $u^n$, $n\ge0$ on both sides of \eqref{mmmn}, we get \eqref{defalstmgp}. This completes the proof.
\end{proof}
\begin{remark}
	Alternatively, by using \eqref{stabrekk}, the system of differential equations given in \eqref{defalstmgp} can be obtained as follows: 
	{\small	\begin{align*}
			\frac{\partial}{\partial t_i}&p^{\boldsymbol{\alpha}}(n,\textbf{t})\\
			&=\frac{\partial}{\partial t_i}\sum_{\Theta(n,d)}\prod_{i=1}^{d}\mathrm{Pr}\{M_i(D_i^{\alpha_i}(t_i))=n_i\}\\
			&=\sum_{\Theta(n,d)}\Big(\prod_{\substack{r=1\\r\ne i}}^{d}\mathrm{Pr}\{M_r(D_r^{\alpha_r}(t_r))=n_r\}\Big)\frac{\partial}{\partial t_i}\mathrm{Pr}\{M_i(D_i^{\alpha_i}(t_i))=n_i\}\\
			&=-\sum_{\Theta(n,d)}\Big(\prod_{\substack{r=1\\r\ne i}}^{d}\mathrm{Pr}\{M_r(D_r^{\alpha_r}(t_r))=n_r\}\Big)\Big(\sum_{j=1}^{k}\lambda_{ji}(I-B^j)\Big)^{\alpha_i}\mathrm{Pr}\{M_i(D_i^{\alpha_i}(t_i))=n_i\}\\
			&=-\Big(\sum_{j=1}^{k}\lambda_{ji}(I-B^j)\Big)^{\alpha_i}p^{\boldsymbol{\alpha}}(n,\textbf{t}),
	\end{align*}}
	where the penultimate step follows on using Eq. (53) of \cite{Kataria2022}. 
\end{remark}

%The proof of following result follows similar lines to that of Theorem 4.4 of \cite{Vishwakarma2025}. Thus, it is omitted.
%\begin{lemma}\label{lemmgcp}
%For $i=1,2,\dots,d$, let $\{N_{i}(D_i^{\alpha_i}(t_i))\}_{t_i\ge0}$, $\alpha_i\in(0,1)$ be independent one parameter counting processes. A multiparameter counting process $\{\mathscr{N}(\boldsymbol{\mathscr{D}}^{\boldsymbol{\alpha}}(\textbf{t}))\}_{\textbf{t}\in\mathbb{R}^d_+}$ defined by $\mathscr{N}(\boldsymbol{\mathscr{D}}^{\boldsymbol{\alpha}}(\textbf{t}))=\sum_{i=1}^{d}N_{i}(D_i^{\alpha_i}(t_i))$, $\textbf{t}=(t_1,t_2,\dots,t_d)\in\mathbb{R}^d_+$ is a multiparameter space fractional Poisson process with transition parameter $\boldsymbol{\Lambda}=(\lambda_{1},\lambda_{2},\dots,\lambda_{d})\succ\textbf{0}$ if and only if  $\{N_{i}(D_i^{\alpha_i}(t_i))\}_{t_i\ge0}$'s are one parameter space fractional Poisson processes with transition rates $\lambda_{i}>0$, respectively.
%\end{lemma}

\begin{remark}
	For $k=1$, the process defined in \eqref{alphmgcp} reduces to a time-changed variant of the multiparameter Poisson process. That is,
	\begin{equation*}
		\mathscr{N}^{\boldsymbol{\alpha}}(\textbf{t})=\mathscr{N}(\boldsymbol{\mathscr{D}}^{\boldsymbol{\alpha}}(\textbf{t})), \ \textbf{t}\in\mathbb{R}^d_+,
	\end{equation*}
	where $\{\mathscr{N}(\textbf{t})\}_{\textbf{t}\in\mathbb{R}^d_+}$ is a multiparameter Poisson process with transition parameter $\boldsymbol{\Lambda}=(\lambda_1,\lambda_2,\dots,\lambda_d)\succ \textbf{0}$ which is independent of  $\{\boldsymbol{\mathscr{D}}^{\boldsymbol{\alpha}}(\textbf{t})\}_{\textbf{t}\in\mathbb{R}^d_+}$. 
	
	Its pgf $\mathcal{G}^{\boldsymbol{\alpha}}(u,\textbf{t})=\mathbb{E}(u^{\mathscr{N}^{\boldsymbol{\alpha}}(\textbf{t})})$, $|u|\le1$ is given by
	\begin{align*}
		\mathcal{G}^{\boldsymbol{\alpha}}(u,\textbf{t})=\exp\Big(-\sum_{i=1}^{d}t_i(\lambda_i(1-u))^{\alpha_i}\Big),
	\end{align*}
	which solves
	\begin{equation*}
		\frac{\partial}{\partial t_i}\mathcal{G}^{\boldsymbol{\alpha}}(u,\textbf{t})=-(\lambda_i(1-u))^{\alpha_i}\mathcal{G}^{\boldsymbol{\alpha}}(u,\textbf{t}),
	\end{equation*}
	with  $\mathcal{G}^{\boldsymbol{\alpha}}(u,\textbf{0})=1$. 
	
	The distribution $q^{\boldsymbol{\alpha}}(n,\textbf{t})=\mathrm{Pr}\{\mathscr{N}^{\boldsymbol{\alpha}}(\textbf{t})=n\}$, $n\ge0$ of multiparameter Poisson process is given by
	\begin{equation*}
		q^{\boldsymbol{\alpha}}(n,\textbf{t})=\sum_{\Theta(n,d)}\prod_{i=1}^{d}\frac{(-1)^{n_i}}{n_i!}\sum_{r_i=0}^{\infty}\frac{(-\lambda_{i}^{\alpha_i}t_i)^{r_i}\Gamma(\alpha_ir_i+1)}{r_i!\Gamma(\alpha_ir_i+1-n_{i})},
	\end{equation*}
	where $\Theta (n,d)=\{(n_{1},n_{2}$, $\dots,n_{d}):\sum_{i=1}^{d}n_{i}=n,\, n_{i}\in\mathbb{N}_0\}$. 
	It solves the following system of differential equations:
	\begin{equation*}
		\frac{\partial}{\partial t_i}q^{\boldsymbol{\alpha}}(n,\textbf{t})=-(\lambda_{i}(I-B))^{\alpha_i}q^{\boldsymbol{\alpha}}(n,\textbf{t}),\ n\ge0, \ \textbf{t}\in\mathbb{R}^d_+,
	\end{equation*} 
	with  $q^{\boldsymbol{\alpha}}(0,\textbf{0})=1$ and $q^{\boldsymbol{\alpha}}(n,\textbf{0})=0$ for all $n\ge 1$.
\end{remark} 
\subsubsection{Time-changed with multivariate stable subordinator} 
Let  $\boldsymbol{\mathscr{D}}^{\boldsymbol{\alpha}}(t)=(D_1^{\alpha_1}(t),D_2^{\alpha_2}(t)$, $\dots, D_d^{\alpha_d}(t))$, $\boldsymbol{\alpha}=(\alpha_1,\alpha_2,\dots,\alpha_d)$, $\alpha_i\in(0,1)$ be a multivariate stable subordinator such that $\{D_i^{\alpha_i}(t)\}_{t\ge0}$'s are independent stable subordinators. We consider  the multiparameter GCP time-changed with a multivariate stable subordinator. That is, 
\begin{equation*}
	\mathscr{M}^{\boldsymbol{\alpha}}(t)\coloneqq\mathscr{M}(\boldsymbol{\mathscr{D}}^{\boldsymbol{\alpha}}(t)),\ t\ge0,
\end{equation*}
where
the multiparameter GCP $\{\mathscr{M}(\textbf{t})\}_{\textbf{t}\in\mathbb{R}^d_+}$ is independent of $\{\boldsymbol{\mathscr{D}}^{\boldsymbol{\alpha}}(t)\}_{t\ge0}$. 

Its Laplace transform is given by $\mathbb{E}(e^{-\eta\mathscr{M}^{\boldsymbol{\alpha}}(t)})=\exp(-t\sum_{i=1}^{d}(\sum_{j=1}^{k}\lambda_{ji}(1-e^{-\eta j}))^{\alpha_i}), \, \eta>0$, and its pgf  $G^{\boldsymbol{\alpha}}(u,t)=\mathbb{E}(u^{\mathscr{M}^{\boldsymbol{\alpha}}(t)})$, $|u|\le 1$ is
\begin{equation*}
	G^{\boldsymbol{\alpha}}(u,t)=\exp\Big(-t\sum_{i=1}^{d}\Big(\sum_{j=1}^{k}\lambda_{ji}(1-u^j)\Big)^{\alpha_i}\Big),
\end{equation*}
which solves the following differential equation:
\begin{equation}\label{pgfmgdeqq}
	\frac{\mathrm{d}}{\mathrm{d}t}G^{\boldsymbol{\alpha}}(u,t)=-\sum_{i=1}^{d}\Big(\sum_{j=1}^{k}\lambda_{ji}(1-u^j)\Big)^{\alpha_i}G^{\boldsymbol{\alpha}}(u,t), \ |u|\le1,
\end{equation}
with $G^{\boldsymbol{\alpha}}(u,0)=1$. 
\begin{remark}
	Note that $\{\mathscr{M}^{\boldsymbol{\alpha}}(t)\}_{t\ge0}$ is a L\'evy process. For $d=1$, it reduces to a space fractional version of the GCP (see \cite[Section 5]{Kataria2022}). For $d=1$ and $k=1$, it reduces to the space fractional Poisson process (see \cite{Orsingher2012}). 
\end{remark}
\begin{remark}\label{alphasgcp}
	For $i=1,2,\dots,d$, let $\{M_i(t_i)\}_{t_i\ge0}$'s be independent GCPs.  From Remark \ref{MGasGcp}, it follows that  $\mathscr{M}^{\boldsymbol{\alpha}}(t)\overset{d}{=}M_1(D_1^{\alpha_1}(t))+$ $M_2(D_2^{\alpha_2}(t))+\dots+M_d(D_d^{\alpha_d}(t))$, where $\{M_i(t_i)\}_{t_i\ge0}$'s are independent of  $\{D_i^{\alpha_i}(t)\}_{t\ge0}$'s. 
\end{remark}

\begin{theorem}
	The distribution $p^{\boldsymbol{\alpha}}(n,t)=\mathrm{Pr}\{\mathscr{M}^{\boldsymbol{\alpha}}(t)=n\}$, $n\ge0$ is given by
	{\small	\begin{equation*}
			p^{\boldsymbol{\alpha}}(n,t)=\sum_{\Theta(n,d)}\prod_{i=1}^{d}\sum_{\Omega(k,n_i)}\Big(\prod_{j=1}^{k}\Big(-\frac{\lambda_{ji}}{\mu_i}\Big)^{n_{ji}}\frac{1}{n_{ji}!}\Big)\sum_{r_i=0}^{\infty}\frac{(-\mu_{i}^{\alpha_i}t)^{r_i}\Gamma(\alpha_ir_i+1)}{r_i!\Gamma(\alpha_ir_i+1-\sum_{j=1}^{k}n_{ji})},
	\end{equation*}}
	where $\mu_i=\lambda_{1i}+\lambda_{2i}+\dots+\lambda_{ki}$, $\Theta (n,d)=\{(n_{1},n_{2}$, $\dots,n_{d}):\sum_{i=1}^{d}n_{i}=n,\, n_{i}\in\mathbb{N}_0\}$ and $\Omega(k,n_i)=\{(n_{1i},n_{2i},\dots,n_{ki}):\sum_{j=1}^{k}jn_{ji}=n_i,\, n_{ji}\in\mathbb{N}_0\}$. It solves the following system of differential equations:
	\begin{equation}\label{defalstmgpm}
		\frac{\mathrm{d}}{\mathrm{d}t}p^{\boldsymbol{\alpha}}(n,t)=-\sum_{i=1}^{d}\Big(\sum_{j=1}^{k}\lambda_{ji}(I-B^j)\Big)^{\alpha_i}p^{\boldsymbol{\alpha}}(n,t),\ n\ge0,
	\end{equation} 
	with $p^{\boldsymbol{\alpha}}(0,0)=1$ and $p^{\boldsymbol{\alpha}}(n,0)=0$ for all $n\ge 1$. Here, $B$ is the backward shift operator, that is, $Bp^{\boldsymbol{\alpha}}(n,t)=p^{\boldsymbol{\alpha}}(n-1,t)$.
\end{theorem}
\begin{proof}
	From Remark \ref{alphasgcp}, we have
	\begin{equation*}
		p^{\boldsymbol{\alpha}}(n,t)=\sum_{\Theta(n,d)}\prod_{i=1}^{d}\mathrm{Pr}\{M_i(D_i^{\alpha_i}(t))=n_i\}.
	\end{equation*}
	By using Eq. (49) of \cite{Kataria2022}, we obtain the required distribution.
	Also, by using \eqref{pgfmgdeqq} and following along the similar lines to that of Theorem \ref{thmalph}, we obtain \eqref{defalstmgpm}. This completes the proof.
\end{proof}
\subsection{Time-changed by multiparameter inverse stable subordinators}
Let $\boldsymbol{\mathscr{L}}_{\boldsymbol{\alpha}}(\textbf{t})=(L_1^{\alpha_1}(t_1),L_2^{\alpha_2}(t_2),\dots,L_d^{\alpha_d}(t_d))$, $\boldsymbol{\alpha}=(\alpha_1,\alpha_2,\dots,\alpha_d)$, $\alpha_i\in(0,1)$ be a multiparameter inverse stable subordinator as defined in (\ref{missdef}) such that its component processes are independent. Let $\{\mathscr{M}(\textbf{t})\}_{\textbf{t}\in\mathbb{R}^d_+}$ be a multiparameter GCP with transition parameters $\boldsymbol{\Lambda}_j=(\lambda_{j1},\lambda_{j2},\dots,\lambda_{jd})\succ\textbf{0}$, $j=1,2,\dots,k$. We consider a multiparameter process $\{\mathscr{M}_{\boldsymbol{\alpha}}(\textbf{t}),\ \textbf{t}\in\mathbb{R}^d_+\}$ defined as follows: 
\begin{equation}\label{mfppdef}
	\mathscr{M}_{\boldsymbol{\alpha}}(\textbf{t})\coloneqq\mathscr{M}(\boldsymbol{\mathscr{L}}_{\boldsymbol{\alpha}}(\textbf{t})),\ \textbf{t}\in\mathbb{R}^d_+,
\end{equation}
where $\{\mathscr{M}(\textbf{t})\}_{\textbf{t}\in\mathbb{R}^d_+}$ is independent of $\{\boldsymbol{\mathscr{L}}_{\boldsymbol{\alpha}}(\textbf{t})\}_{\textbf{t}\in\mathbb{R}^d_+}$. We call it the multiparameter fractional GCP with transition parameters $\boldsymbol{\Lambda}_j$. 
\begin{remark}
	For $d=1$ and $k=1$, it reduces to the generalized fractional counting process (see \cite{Kataria2022a}) and the multiparameter fractional Poisson process (see \cite{Vishwakarma2025}), respectively. Also, for $d=k=1$, it reduces to the time fractional Poisson process (see \cite{Meerschaert2011}).
\end{remark}
For $|u|\le 1$, the pgf of $\{\mathscr{M}_{\boldsymbol{\alpha}}(\textbf{t})\}_{\textbf{t}\in\mathbb{R}^d_+}$ is given by
{\small\begin{align}
		G_{\boldsymbol{\alpha}}(u,\textbf{t})&=\mathbb{E}(\mathbb{E}(u^{\mathscr{M}_{\boldsymbol{\alpha}}(\textbf{t})}|\boldsymbol{\mathscr{L}}_{\boldsymbol{\alpha}}(\textbf{t})))\nonumber\\
		&=\mathbb{E}\Big(\exp\Big(-\sum_{j=1}^{k}\boldsymbol{\Lambda}_{j}\cdot\boldsymbol{\mathscr{L}}_{\boldsymbol{\alpha}}(\textbf{t})(1-u^j)\Big)\Big),\ \text{(using \eqref{MPGPGF})}\nonumber\\
		&=\prod_{i=1}^{d}\mathbb{E}\Big(\exp\Big(-\sum_{j=1}^{k}\lambda_{ji}L_i^{\alpha_i}(t_i)(1-u^j)\Big)\Big),\ (\text{as $L_i^{\alpha_{i}}(t_i)$'s are independent})\nonumber\\
		&=\prod_{i=1}^{d}E_{\alpha_i,1}\Big(-t_i^{\alpha_i}\sum_{j=1}^{k}\lambda_{ji}(1-u^j)\Big),\ \textbf{t}\in\mathbb{R}^d_+,\label{pgfinvMgcp}
\end{align}}
where the last step follows on using the Laplace transform of inverse stable subordinator.
\begin{remark}
	The $n$th factorial moment of $\{\mathscr{M}_{\boldsymbol{\alpha}}(\textbf{t})\}_{\textbf{t}\in\mathbb{R}^d_+}$ is
	\begin{align*}
		\mathbb{E}&(\mathscr{M}_{\boldsymbol{\alpha}}(\textbf{t})(\mathscr{M}_{\boldsymbol{\alpha}}(\textbf{t})-1)\dots(\mathscr{M}_{\boldsymbol{\alpha}}(\textbf{t})-n+1))\\
		&=\frac{\partial^n}{\partial u^n}G_{\boldsymbol{\alpha}}(u,\textbf{t})|_{u=1}\\
		&=\sum_{\Theta(n,d)}\frac{n!}{n_1!n_2!\dots n_d!}\prod_{i=1}^{d}\frac{\partial^{n_i}}{\partial u^{n_i}}E_{\alpha_i,1}\Big(-t_i^{\alpha_i}\sum_{j=1}^{k}\lambda_{ji}(1-u^j)\Big)\\
		&=\sum_{\Theta(n,d)}n!\prod_{i=1}^{d}\sum_{r_i=1}^{n_i}\frac{t_i^{r_i\alpha_i}}{\Gamma(r_i\alpha_i+1)}\sum_{n_{i1}+n_{i2}+\dots+n_{ir_i}=n_i}\prod_{l_i=1}^{r_i}\sum_{j=1}^{k}\frac{\lambda_{ji}(j)_{n_{il_i}}}{n_{il_i}!},
	\end{align*}
	where $\Theta(n,d)=\{(n_1,n_2,\dots,n_d):\sum_{i=1}^{d}n_i=n,\, n_i\in\mathbb{N}_0\}$ and $(n)_r=n(n-1)\dots (n-r+1)$.
\end{remark}
\begin{theorem}
	The distribution $p_{\boldsymbol{\alpha}}(n,\textbf{t})=\mathrm{Pr}\{\mathscr{M}_{\boldsymbol{\alpha}}(\textbf{t})=n\}$, $n\ge0$ of multiparameter fractional GCP is given by
	\begin{align}\label{mfppdist}
		p_{\boldsymbol{\alpha}}(n,\textbf{t})
		&=
		\sum_{\Omega(k,n)}	\sum_{\substack{\Theta(n_j,d)\\j=1,2,\dots,k}}\prod_{i=1}^{d}\Big(\sum_{j=1}^{k}n_{ji}\Big)!\Big(\prod_{j=1}^{k}\frac{(\lambda_{ji}t_i^{\alpha_i})^{n_{ji}}}{n_{ji}!}\Big)\nonumber\\
		&\hspace{3.8cm}\cdot E_{\alpha_i,\, \alpha_i\sum_{j=1}^{k}n_{ji}+1}^{\sum_{j=1}^{k}n_{ji}+1}\Big(-t_i^{\alpha_i}\sum_{j=1}^{k}\lambda_{ji} \Big),
	\end{align}
	where $\Omega(k,n)=\{(n_1,n_2,\dots,n_k):\sum_{j=1}^{k}jn_j=n,\, n_j\in\mathbb{N}_0\}$ and $\Theta(n_j,d)=$ $\{(n_{j1},n_{j2},\dots$, $n_{jd}): 0\leq n_{ji}\leq n_j,\, \sum_{i=1}^{d}n_{ji}=n_j\}$ and $E_{\alpha,\,\beta}^\gamma(\cdot)$ is the three parameter Mittag-Leffler function defined in \eqref{mitag}.
	Also, for $i=1,2,\dots,d$, the distribution (\ref{mfppdist}) solves the following system of fractional differential equations:
	\begin{equation}\label{sysdeinvMGCP}
		\mathcal{D}_{t_i}^{\alpha_i}p_{\boldsymbol{\alpha}}(n,\textbf{t})=-\sum_{j=1}^{k}\lambda_{ji}(p_{\boldsymbol{\alpha}}(n,\textbf{t})-p_{\boldsymbol{\alpha}}(n-j,\textbf{t})),\ n\ge0,
	\end{equation}
	where $\mathcal{D}_{t_i}^{\alpha_i}$ is the Caputo fractional derivative defined in \eqref{caputo}.
\end{theorem}
\begin{proof}
	Let $l_{\alpha_i}(x_i,t_i)$, $x_i\ge0$ be the density of $L_i^{\alpha_i}(t_i)$. Then, by using (\ref{mfppdef}), the distribution of multiparameter fractional GCP can be expressed as  
	\begin{equation}\label{mfppdistpf1}
		p_{\boldsymbol{\alpha}}(n,\textbf{t})=\int_{\mathbb{R}^d_+}p(n,\textbf{x})\prod_{i=1}^{d}l_{\alpha_i}(x_i,t_i)\, \mathrm{d}x_i.
	\end{equation} 
	Also, let $\textbf{w}=(w_1,w_2,\dots,w_d)\succ\textbf{0}$. Then, by using (\ref{lapminv}) and (\ref{MGCPMF}), its Laplace transform can be written as
	{\small	\begin{align}
			\int_{\mathbb{R}^d_+}&e^{-\textbf{w}\cdot\textbf{t}}p_{\boldsymbol{\alpha}}(n,\textbf{t})\,\mathrm{d}\textbf{t}\nonumber\\
			&=\sum_{\Omega(k,n)}\int_{\mathbb{R}^d_+}\Big(\prod_{j=1}^{k}\frac{(\boldsymbol{\Lambda}_j\cdot\textbf{x})^{n_j}}{n_j!}e^{-\boldsymbol{\Lambda}_j\cdot\textbf{x}}\Big)\prod_{i=1}^{d}
			w_i^{\alpha_i-1}e^{-w_i^{\alpha_i}x_i}\,\mathrm{d}x_i\nonumber\\
			&=\sum_{\Omega(k,n)}\sum_{\substack{\Theta(n_j,d)\\j=1,2,\dots,k}}\Big(\prod_{j=1}^{k}\prod_{i=1}^{d}\frac{\lambda_{ji}^{n_{ji}}}{n_{ji}!}\Big)\int_{\mathbb{R}^d_+}\Big(\prod_{j=1}^{k}\prod_{i=1}^{d}x_i^{n_{ji}}e^{-\lambda_{ji}x_i}\Big)\prod_{i=1}^{d}	w_i^{\alpha_i-1}e^{-w_i^{\alpha_i}x_i}\, \mathrm{d}x_i\nonumber\\	
			&=\sum_{\Omega(k,n)}\sum_{\substack{\Theta(n_j,d)\\j=1,2,\dots,k}}\Big(\prod_{j=1}^{k}\prod_{i=1}^{d}\frac{\lambda_{ji}^{n_{ji}}}{n_{ji}!}\Big)\prod_{i=1}^{d}w_i^{\alpha_i-1}\int_{0}^{\infty}x_i^{\sum_{j=1}^{k}n_{ji}}e^{-\sum_{j=1}^{k}\lambda_{ji}+w_i^{\alpha_i}}\, \mathrm{d}x_i \nonumber\\
			&=\sum_{\Omega(k,n)}\sum_{\substack{\Theta(n_j,d)\\j=1,2,\dots,k}}\Big(\prod_{j=1}^{k}\prod_{i=1}^{d}\frac{\lambda_{ji}^{n_{ji}}}{n_{ji}!}\Big)\prod_{i=1}^{d}\frac{w_i^{\alpha_i-1}\Gamma(\sum_{j=1}^{k}n_{ji}+1)}{(w_i^{\alpha_i}+\sum_{j=1}^{k}\lambda_{ji})^{\sum_{j=1}^{k}n_{ji}+1}}.\label{mdist1}
	\end{align}}
	On taking the inverse Laplace transform on both sides of (\ref{mdist1}) and using Eq. (1.9.13) of \cite{Kilbas2006}, we get \eqref{mfppdist}.
	
	Now, on applying the Riemann-Liouville fractional derivative on both sides of (\ref{mfppdistpf1}) and using (\ref{invsequ}), we have
	\begin{align}
		\partial_{t_i}^{\alpha_i}p_{\boldsymbol{\alpha}}(n,\textbf{t})&=-\int_{\mathbb{R}^d_+}p(n,\textbf{x})\partial_{x_i}l_{\alpha_i}(x_i,t_i)\,\mathrm{d}x_i\prod_{r\ne i}l_{\alpha_r}(x_r,t_r)\,\mathrm{d}x_r\nonumber\\
		&=-\int_{\mathbb{R}^d_+}\partial_{x_i}(p(n,\textbf{x})l_{\alpha_i}(x_i,t_i))\,\mathrm{d}x_i\prod_{r\ne i}l_{\alpha_r}(x_r,t_r)\,\mathrm{d}x_r\nonumber\\
		&\ \ +\int_{\mathbb{R}^d_+}\partial_{x_i}p(n,\textbf{x})l_{\alpha_i}(x_i,t_i)\,\mathrm{d}x_i\prod_{r\ne i}l_{\alpha_r}(x_r,t_r)\,\mathrm{d}x_r\nonumber\\
		&=\frac{t_i^{-\alpha_i}}{\Gamma(1-\alpha_i)}\int_{\mathbb{R}^{d-1}_+}p(n,\textbf{x}^i(0))\prod_{r\ne i}l_{\alpha_r}(x_r,t_r)\,\mathrm{d}x_r\nonumber\\
		&\ \ +\int_{\mathbb{R}^d_+}\partial_{x_i}p(n,\textbf{x})l_{\alpha_i}(x_i,t_i)\,\mathrm{d}x_i\prod_{r\ne i}l_{\alpha_r}(x_r,t_r)\,\mathrm{d}x_r,\label{cpf0}
	\end{align}
	where $\textbf{x}^i(0)=(x_1,x_2,\dots,x_{i-1},0,x_{i+1},\dots,x_d)\in\mathbb{R}^d_+$.
	
	In (\ref{cpf0}), we use
	$\partial_{t_i}p(n,\textbf{t})=-\sum_{j=1}^{k}\lambda_{ji}(p(n,\textbf{t})-p(n-j,\textbf{t}))$, $n\ge0$, $i=1,2,\dots,d$
	and
	\begin{equation*}
		\int_{\mathbb{R}^{d-1}_+}p(n,\textbf{x}^i(0))\prod_{r\ne i}l_{\alpha_r}(x_r,t_r)\,\mathrm{d}x_r=p_{\boldsymbol{\alpha}}(n,\textbf{t}^i(0))
	\end{equation*}
	to obtain
	\begin{equation*}
		\partial_{t_i}^{\alpha_i}p_{\boldsymbol{\alpha}}(n,\textbf{t})=-\sum_{j=1}^{k}\lambda_{ji}(p_{\boldsymbol{\alpha}}(n,\textbf{t})-p_{\boldsymbol{\alpha}}(n-j,\textbf{t}))+\frac{t_i^{-\alpha_i}}{\Gamma(1-\alpha_i)}p_{\boldsymbol{\alpha}}(n,\textbf{t}^i(0)).
	\end{equation*}
	This completes the proof.
\end{proof}
\begin{theorem}\label{summfppthm}
	For $i=1,2,\dots,d$ and $\alpha_i\in(0,1)$, let $\{M_1^{\alpha_1}(t_1),\ t_1\ge0\}$, $\{M_2^{\alpha_2}(t_2),\ t_2\ge0\}$, $\dots$, $\{M_d^{\alpha_d}(t_d),\ t_d\ge0\}$ be independent one parameter counting processes. Then, a multiparameter counting process $\{\mathscr{M}_{\boldsymbol{\alpha}}(\textbf{t})\}_{\textbf{t}\in\mathbb{R}^d_+}$ defined by $\mathscr{M}_{\boldsymbol{\alpha}}(\textbf{t})\coloneqq\sum_{i=1}^{d}M_i^{\alpha_i}(t_i)$, $t_i\ge0$ is a multiparameter fractional GCP with transition parameters $\boldsymbol{\Lambda}_j=(\lambda_{j1},\lambda_{j2},\dots,\lambda_{jd})\succ\textbf{0}$, $j=1,2,\dots,k$ if and only if for each $i=1,2,\dots,d$, $\{M_i^{\alpha_i}(t_i)\}_{t_i\ge0}$ is a generalized fractional counting process with transition rates $\lambda_{ji}>0$, $j=1,2,\dots,k$.
\end{theorem} 
\begin{proof}
	In view of Theorem \ref{MGasGcp}, the proof follows along the similar lines to that of Theorem 4.4 of \cite{Vishwakarma2025}. Hence, it is omitted.
\end{proof}

\begin{proposition}\label{prpinv}
	Let $\{M_i(t_i)\}_{t_i\ge0}$, $i=1,2,\dots,d$ be independent GCPs. Also, let $\boldsymbol{\mathscr{L}}_{\boldsymbol{\alpha}}(\textbf{t})=(L_1^{\alpha_1}(t_1),L_2^{\alpha_2}(t_2),\dots,L_d^{\alpha_d}(t_d))$, $\textbf{t}\in\mathbb{R}^d_+$ be a multiparameter inverse stable subordinator as defined in (\ref{missdef}). Then, 
	\begin{equation}\label{MInvGaGCp}
		\mathscr{M}_{\boldsymbol{\alpha}}(\textbf{t})\overset{d}{=}\sum_{i=1}^{d}M_i(L_{i}^{\alpha_i}(t_i)),
	\end{equation}
	where $\{M_i(t_i)\}_{t_i\ge0}$'s are independent of the inverse stable subordinators $\{L_i^{\alpha_i}$ $(t_i)\}_{t_i\ge0}$'s.
\end{proposition}
\begin{proof}
	As $\{M_i(t_i)\}_{t_i\ge0}$'s are independent GCPs and $\{L_i^{\alpha_i}(t_i)\}_{t_i\ge0}$'s are independent inverse stable subordinators,  the component processes $\{M_i(L_i^{\alpha}(t_i))\}_{t_i\ge0}$ in the right hand side of \eqref{MInvGaGCp} are independent generalized fractional counting processes (see \cite{Di Crescenzo2016}). Therefore, by using Eq. (14) of \cite{Kataria2022a}, the pgf of $\{\sum_{i=1}^{d}M_i$ $(L_i^{\alpha_i}(t_i))\}_{t_i\ge0}$ is given by
	{\small	\begin{align*}
			\mathbb{E}(u^{\sum_{i=1}^{d}M_i(L_i^{\alpha_i}(t_i))})&=\prod_{i=1}^{d}\mathbb{E}(u^{M_i(L_i^{\alpha_i}(t_i))})
			=\prod_{i=1}^{d}E_{\alpha_i,1}\Big(-t_i^{\alpha_i}\sum_{j=1}^{k}\lambda_{ji}(1-u^j)\Big),\ |u|\le1,
	\end{align*}}
	which coincides with \eqref{pgfinvMgcp}. This completes the proof.
\end{proof}
\begin{remark}\label{reminv}
	By using \eqref{MInvGaGCp} and Eq. (2.8) of \cite{Di Crescenzo2016}, the distribution function 
	$p_{\boldsymbol{\alpha}}(n,\textbf{t})$
	$=\mathrm{Pr}\{\mathscr{M}_{\boldsymbol{\alpha}}(\textbf{t})=n\}$, $n\ge0$ given in \eqref{mfppdist} can be obtained in the equivalent form:
	{\small	\begin{align}\label{MGInAsGdis}
			p_{\boldsymbol{\alpha}}(n,\textbf{t})&=\sum_{\Theta(n,d)}\prod_{i=1}^{d}\mathrm{Pr}\{M_i(L_i^{\alpha_i}(t_i))=n_i\}\\
			&=\sum_{\Theta(n,d)}	\sum_{\substack{\Omega(k,n_i)\\i=1,2,\dots,d}}\prod_{i=1}^{d}\Big(\sum_{j=1}^{k}n_{ji}\Big)!\Big(\prod_{j=1}^{k}\frac{(\lambda_{ji}t_i^{\alpha_i})^{n_{ji}}}{n_{ji}!}\Big)E_{\alpha_i,\, \alpha_i\sum_{j=1}^{k}n_{ji}+1}^{\sum_{j=1}^{k}n_{ji}+1}\Big(-t_i^{\alpha_i}\sum_{j=1}^{k}\lambda_{ji} \Big),\nonumber
	\end{align}}
	where  $\Omega(k,n_i)=\{(n_{1i},n_{2i},\dots,n_{ki}):\sum_{j=1}^{k}jn_{ji}=n_i,\, n_{ji}\in\mathbb{N}_0\}$ and $\Theta(n,d)=$ $\{(n_{1},n_{2},\dots$, $n_{d}): 0\leq n_{i}\leq n,\, \sum_{i=1}^{d}n_{i}=n\}$. 
	
	Also, by using \eqref{MGInAsGdis}, the system of differential equations given in \eqref{sysdeinvMGCP} can be obtained as follows:
	{\small	\begin{align*}
			\mathcal{D}_{t_i}^{\alpha_i}p_{\boldsymbol{\alpha}}(n,\textbf{t})&=\mathcal{D}_{t_i}^{\alpha_i}\sum_{\Theta(n,d)}\prod_{i=1}^{d}\mathrm{Pr}\{M_i(L_i^{\alpha_i}(t_i))=n_i\}\\
			&=\sum_{\Theta(n,d)}\Big(\prod_{\substack{r=1\\r\ne i}}^{d}\mathrm{Pr}\{M_r(L_r^{\alpha_r}(t_r))=n_r\}\Big)\mathcal{D}_{t_i}^{\alpha_i}\mathrm{Pr}\{M_i(L_i^{\alpha_i}(t_i))=n_i\}\\
			&=\sum_{\Theta(n,d)}\Big(\prod_{\substack{r=1\\r\ne i}}^{d}\mathrm{Pr}\{M_r(L_r^{\alpha_r}(t_r))=n_r\}\Big)\Big(-\sum_{j=1}^{k}\lambda_{ji}\mathrm{Pr}\{M_i(D_i^{\alpha_i}(t_i))=n_i\}\\
			&\hspace{4cm}+\sum_{j=1}^{k}\lambda_{ji}\mathrm{Pr}\{M_i(D_i^{\alpha_i}(t_i))=n_i-j\}\Big)\\
			&=      -\sum_{j=1}^{k}\lambda_{ji}(p_{\boldsymbol{\alpha}}(n,\textbf{t})-p_{\boldsymbol{\alpha}}(n-j,\textbf{t})),\ n\ge0,
	\end{align*}}
	where the penultimate step follows on using Eq. (2.3) of \cite{Di Crescenzo2016}. 
\end{remark}

Thus, by using \eqref{MInvGaGCp} and Remark 2.11 of \cite{Di Crescenzo2016}, the mean and variance of multiparameter fractional GCP are 
\begin{equation*}
	\mathbb{E}(\mathscr{M}_{\boldsymbol{\alpha}}(\textbf{t}))=\sum_{i=1}^{d}\sum_{j=1}^{k}\frac{j\lambda_{ji}t_i^{\alpha_i}}{\Gamma(\alpha_i+1)}
\end{equation*}
and
{\small\begin{equation*}
		\operatorname{Var}(\mathscr{M}_{\boldsymbol{\alpha}}(\textbf{t}))=\sum_{i=1}^{d}\Big(\sum_{j=1}^{k}\frac{j^2\lambda_{ji}t_i^{\alpha_i}}{\Gamma(\alpha_i+1)}+\Big(\sum_{j=1}^{k}j\lambda_{ji}t_i^{\alpha_i}\Big)^2\Big(\frac{2}{\Gamma(2\alpha_i+1)}-\frac{1}{\Gamma^2(\alpha_i+1)}\Big)\Big),
\end{equation*}}
respectively.

\subsubsection{Time-changed with multivariate inverse stable subordinator} 
Let $\boldsymbol{\mathscr{L}}_{\boldsymbol{\alpha}}(t)=(L_1^{\alpha_1}(t)$, $L_2^{\alpha_2}(t),\dots,L_d^{\alpha_d}(t))$ be a multivariate inverse stable subordinator such that the component processes are independent inverse stable subordinators. Then, we consider the following time-changed variant of multiparameter GCP:
\begin{equation*}
	\mathscr{M}_{\boldsymbol{\alpha}}(t)\coloneqq\mathscr{M}(\boldsymbol{\mathscr{L}}_{\boldsymbol{\alpha}}(t)),\ t\ge0,
\end{equation*}
where
the multiparameter GCP $\{\mathscr{M}(\textbf{t})\}_{\textbf{t}\in\mathbb{R}^d_+}$ is independent of $\{\boldsymbol{\mathscr{L}}_{\boldsymbol{\alpha}}(t)\}_{t\ge0}$. 

Its pgf  $G_{\boldsymbol{\alpha}}(u,t)=\mathbb{E}(u^{\mathscr{M}_{\boldsymbol{\alpha}}(t)})$, $|u|\le 1$ is
\begin{equation*}
	G_{\boldsymbol{\alpha}}(u,t)=\prod_{i=1}^{d}E_{\alpha_i,1}\Big(-t^{\alpha_i}\sum_{j=1}^{k}\lambda_{ji}(1-u^j)\Big).
\end{equation*} 
\begin{remark}
	For $d=1$, $\{\mathscr{M}_{\boldsymbol{\alpha}}(t)\}_{t\ge0}$ reduces to the generalized fractional counting process (see \cite{Di Crescenzo2016}, \cite{Kataria2022a}). For $k=1$, it reduces to the multiparameter Poisson process time-changed with multivariate inverse stable subordinator. For $d=k=1$, it reduces to the time fractional Poisson process (see \cite{Beghin2009}). 
\end{remark}
\begin{remark}\label{invalphasgcp}
	Let $\{M_i(t_i)\}_{t_i\ge0}$, $i=1,2,\dots,d$ be independent GCPs.  Then, following along the similar lines to that of Proposition \ref{prpinv}, it can be established that  $\mathscr{M}_{\boldsymbol{\alpha}}(t)\overset{d}{=}\sum_{i=1}^{d}M_i(L_i^{\alpha_i}(t))$, where $\{M_i(t_i)\}_{t_i\ge0}$'s are independent of  $\{L_i^{\alpha_i}(t)\}_{t\ge0}$'s. 
\end{remark}
\begin{proposition}
	The distribution $p_{\boldsymbol{\alpha}}(n,t)=\mathrm{Pr}\{\mathscr{M}_{\boldsymbol{\alpha}}(t)=n\}$, $n\ge0$ is given by
	{\small	\begin{equation*}
			p_{\boldsymbol{\alpha}}(n,t)=\sum_{\Theta(n,d)}\prod_{i=1}^{d}\sum_{\Omega(k,n_i)}\Big(\prod_{j=1}^{k}\Big(-\frac{\lambda_{ji}}{\mu_i}\Big)^{n_{ji}}\frac{1}{n_{ji}!}\Big)\sum_{r_i=0}^{\infty}\frac{(-\mu_{i}^{\alpha_i}t)^{r_i}\Gamma(\alpha_ir_i+1)}{r_i!\Gamma(\alpha_ir_i+1-\sum_{j=1}^{k}n_{ji})},
	\end{equation*}}
	where $\mu_i=\lambda_{1i}+\lambda_{2i}+\dots+\lambda_{ki}$, $\Theta (n,d)=\{(n_{1},n_{2}$, $\dots,n_{d}):\sum_{i=1}^{d}n_{i}=n,\, n_{i}\in\mathbb{N}_0\}$ and $\Omega(k,n_i)=\{(n_{1i},n_{2i},\dots,n_{ki}):\sum_{j=1}^{k}jn_{ji}=n_i,\, n_{ji}\in\mathbb{N}_0\}$. 
\end{proposition}
\begin{proof}
	From Remark \ref{invalphasgcp}, we have
	\begin{equation*}
		p_{\boldsymbol{\alpha}}(n,t)=\sum_{\Theta(n,d)}\prod_{i=1}^{d}\mathrm{Pr}\{M_i(L_i^{\alpha_i}(t))=n_i\}.
	\end{equation*}
	By using Eq. (2.8) of \cite{Di Crescenzo2016}, we obtain the required distribution. 
\end{proof}
From Remark \ref{invalphasgcp}, the mean and variance of $\{\mathscr{M}_{\boldsymbol{\alpha}}(t)\}_{t\ge0}$ are 
\begin{equation*}
	\mathbb{E}(\mathscr{M}_{\boldsymbol{\alpha}}(t))=\sum_{i=1}^{d}\sum_{j=1}^{k}\frac{j\lambda_{ji}t^{\alpha_i}}{\Gamma(\alpha_i+1)}
\end{equation*}
and
{\small\begin{equation*}
		\operatorname{Var}(\mathscr{M}_{\boldsymbol{\alpha}}(t))=\sum_{i=1}^{d}\Big(\sum_{j=1}^{k}\frac{j^2\lambda_{ji}t^{\alpha_i}}{\Gamma(\alpha_i+1)}+\Big(\sum_{j=1}^{k}j\lambda_{ji}t^{\alpha_i}\Big)^2\Big(\frac{2}{\Gamma(2\alpha_i+1)}-\frac{1}{\Gamma^2(\alpha_i+1)}\Big)\Big),
\end{equation*}}
respectively.

\section*{Acknowledgement}
The first author thanks Government of India for the grant of Prime Minister's Research Fellowship, ID 1003066.

\end{document}